\documentclass{amsart}
\usepackage{amssymb,amsmath,latexsym,enumerate, dsfont}
\usepackage{graphicx,verbatim,enumerate,%
ifpdf,mdwlist}
\usepackage{color}
\usepackage{mathabx}
\ifpdf
\usepackage{hyperref}
\else
\usepackage[hypertex]{hyperref}
\fi
\usepackage{graphicx}
\usepackage{cancel}
\usepackage{float}

\renewcommand{\phi}{\varphi}

\newtheorem{theorem}{Theorem}[section]

\newtheorem{lemma}[theorem]{Lemma}
\newtheorem{corollary}[theorem]{Corollary}
\newtheorem{proposition}[theorem]{Proposition}
\newtheorem{question}[theorem]{Question}

\newtheorem{defi}[theorem]{Definition}
\newenvironment{emdef}{\begin{defi} \rm}{ \end{defi}}
\newtheorem{exa}[theorem]{Example}

\newenvironment{remark}{\begin{rem} \rm}{ \end{rem}}

\newtheorem{rem}[theorem]{Remark}
\newtheorem{claim}[theorem]{Claim}

\DeclareMathOperator{\range}{range}

\DeclareMathOperator{\Id}{Id}

\newcommand{\ZeU}{\#\left(0^{U}\right)}
\newcommand{\ZeUS}{\#\left(0^{\sigma^U_s}\right)}
\newcommand{\ZeX}{\#\left(0^{X}\right)}
\newcommand{\ZeY}{\#\left(0^{Y}\right)}
\newcommand{\ZeZ}{\#\left(0^{Z}\right)}

\newcommand{\ZeYz}{\#\left(0^{Y_{0}}\right)}
\newcommand{\ZeYu}{\#\left(0^{Y_{1}}\right)}

\newcommand{\ZeXs}{\#\left(0^{X^{*}}\right)}
\newcommand{\ZeYs}{\#\left(0^{Y^{*}}\right)}

\newcommand{\ZYv}{\#\left(0^{Y_{0} \lor Y_{1}}\right)}
\newcommand{\ZYw}{\#\left(0^{Y_{0} \wedge Y_{1}}\right)}

\DeclareMathOperator{\Peq}{\mathbf{Peq}}
\DeclareMathOperator{\ZeS}{\#\left(0^{\sigma}\right)}
\DeclareMathOperator{\ZeSz}{\#\left(0^{\sigma_0}\right)}
\DeclareMathOperator{\ZeSu}{\#\left(0^{\sigma_1}\right)}
\DeclareMathOperator{\ZeA}{\#\left(0^{\tau}\right)}
\DeclareMathOperator{\ZeAz}{\#\left(0^{\tau_0}\right)}
\DeclareMathOperator{\ZeAu}{\#\left(0^{\tau_1}\right)}
\DeclareMathOperator{\ZeSlA}{\#\left(0^{\sigma \lor \tau}\right)}
\DeclareMathOperator{\ZeSwA}{\#\left(0^{\sigma \wedge \tau}\right)}

\newcommand{\ZOP}[1]{\#\left(0^{#1}\right)}

\newcommand{\set}[1]{\{{#1}\}}
\newcommand{\rel}[1]{\mathrel{#1}}

\newcommand{\str}[1]{\langle #1 \rangle}
\newcommand{\catt}[2]{#1\widehat{\phantom{\alpha}}\str{#2}}

\title[Punctual equivalence relations]{Punctual equivalence
relations and their (punctual) complexity}

\author[N.~Bazhenov]{Nikolay Bazhenov}
\address{Sobolev Institute of Mathematics, pr. Akad.
Koptyuga 4, Novosibirsk, 630090 Russia
}
\email{bazhenov@math.nsc.ru}

\author[K.M.~Ng]{Keng Meng Ng}
\address{School of Physical and Mathematical Sciences
\\
Nanyang Technological University\\
Singapore}\email{kmng@ntu.edu.sg}

\author[L.~San Mauro]{Luca San Mauro}
\address{Department of Mathematics ``Guido Castelnuovo'', Sapienza University of Rome,
I-00185 Rome, Italy}
\email{\href{mailto:luca.sanmauro@gmail.com}{luca.sanmauro@gmail.com}}

\author[A.~Sorbi]{Andrea Sorbi}
\address{Department of Information Engineering and Mathematics\\
University of Siena\\
I-53100 Siena, Italy}
\email{\href{mailto:andrea.sorbi@unisi.it}{andrea.sorbi@unisi.it}}

\thanks{Bazhenov was supported by the Mathematical Center in Akademgorodok
under agreement No.~075-15-2019-1613
with the Ministry of Science and Higher Education of the Russian Federation.
Ng was supported by MOE Tier 1 grant RG23/19.
San Mauro was supported by the Austrian Science Fund
FWF, Project M 2461. Sorbi is a member of
INDAM-GNSAGA, and was partially supported by PRIN 2017
Grant ``Mathematical Logic: models, sets, computability''.}



\begin{document}

\begin{abstract}
The complexity of equivalence relations has received much attention in the
recent literature. The main tool for such endeavour is the following
reducibility: given equivalence relations $R$ and $S$ on natural numbers,
$R$ is computably reducible to $S$ if there is a computable function $f :
\omega \rightarrow \omega$ that induces an injective map from
$R$-equivalence classes to $S$-equivalence classes. In order to compare the
complexity of equivalence relations which are computable, researchers
considered also feasible variants of computable reducibility, such as the
polynomial-time reducibility. In this work, we explore $\mathbf{Peq}$, the
degree structure generated by primitive recursive reducibility on punctual
equivalence relations (i.e., primitive recursive equivalence relations with
domain $\omega$). In contrast with all other known degree structures on
equivalence relations, we show that $\mathbf{Peq}$ has much more structure:
e.g., we show that it is a dense distributive lattice. On the other hand,
we also offer  evidence of the intricacy of $\mathbf{Peq}$, proving, e.g.,
that the structure is neither rigid nor homogeneous.
\end{abstract}

\maketitle

\section{Introduction}
The classification of equivalence relations according to their complexity is
a major research thread in logic. The following two examples stand out from
the existing literature.
\begin{itemize}
\item In descriptive set theory, one often deals with equivalence relations
    defined on Polish spaces (e.g., $2^\omega$ or $\omega^\omega$), which
    are classified in terms of \emph{Borel embeddings}. The corresponding
    theory is now a consolidated  field of modern descriptive set theory,
    which shows deep connections with topology, group theory,
    combinatorics, model theory, and ergodic theory  (see, e.g.,
    \cite{friedman1989borel,HKL,gao2008invariant,hjorth2010borel}).
\item On the other hand, in computability theory,  it is common to
    concentrate on equivalence relations on the natural numbers and compare
    their  algorithmic content in terms of \emph{computable reductions}
    (see, e.g.,
    \cite{Ershov:Book,Ershov:survey,Gao,Andrews-et-al-14,coskey2012hierarchy}).
\end{itemize}

Computable reducibility (for which we use the symbol $\leq_c$, and call
\emph{$c$-degrees} the elements of the corresponding degree structure)  has
been adopted to calculate the complexity of natural equivalence relations on
$\omega$, proving, e.g., that provable equivalence in Peano Arithmetic is
$\Sigma^0_1$ complete~\cite{Bernardi:83}, Turing equivalence on c.e. sets is
$\Sigma^0_4$ complete~\cite{IMNN-14}, and the isomorphism relations on
several familiar classes of computable structures (e.g., trees, torsion
abelian groups, fields of characteristic $0$ or $p$) are $\Sigma^1_1$
complete~\cite{FFH-12}. In parallel, there has been a growing interest in the
abstract study of the poset of degrees generated by computable reducibility
on the collection of equivalence relations of a certain complexity. Most
notably, the poset \textbf{Ceers} of the $c$-degrees of computably enumerable
equivalence relations (commonly known by the acronym \emph{ceers})  has been
thoroughly explored~\cite{Andrews-Sorbi-19,Andrews-Badaev-Sorbi}: e.g., it
has been recently shown that its first-order theory is as complicated as true
arithmetic~\cite{Andrews-et-al-20}. Less is known about larger structures of
$c$-degrees; but recent studies  considered the $\Delta^0_2$
case~\cite{Ng,Bazhenov-et-al} and the global structure $\mathbf{ER}$ of
\emph{all} $c$-degrees~\cite{andrews2021structure}.

Yet, despite its  classificatory power, computable reducibility has an
obvious shortcoming: it is simply too coarse for measuring the relative
complexity of computable equivalence relations. Indeed, it is easy to note
that any two computable equivalence relations $R$ and $S$ with infinitely
many classes are computably bi-reducible. A natural way to overcome this
limitation is by adopting \emph{feasible} reducibilities, as is done in
\cite{buss2011strong,gao2017polynomial}, where the authors prove that,
relatively to these reducibilities, the isomorphism relations of finite
fields, finite abelian groups, and finite linear orders have all the same
complexity.

This paper focuses on a subcollection of computable equivalence relations,
namely primitive recursive equivalence relations with domain $\omega$, called
\emph{punctual}. To classify punctual equivalence relations, we adopt
primitive recursive reducibility.  More precisely, a punctual equivalence
relation $R$ is \emph{$pr$-reducible} to a punctual equivalence relation $S$
(notation: $R\leq_{pr} S$) if there exists a primitive recursive function $f$
such that
\[
(\forall x, y\in\omega)[x \rel{R} y \Leftrightarrow f(x) \rel{S} f(y)].
\]
So, the main object of study of this paper is $\mathbf{Peq}$, i.e., the
degree structure of $pr$-degrees consisting of punctual equivalence
relations.

In the sequel, we hope to convince the reader that $\mathbf{Peq}$ is a
remarkable structure. For instance, in constrast with $\mathbf{Ceers}$ and
all other established structures of $c$-degrees, $\mathbf{Peq}$ is
surprisingly well-behaved, being a dense distributive lattice (Theorems
\ref{theor:lattice} and \ref{thm:density}). On the other hand, we also offer
evidence of the intricacy of $\mathbf{Peq}$, proving, e.g., that the
structure is neither rigid nor homogeneous (Theorems \ref{thm:nonrigidity}
and \ref{cor:finitenotdefinable}). Furthermore, working in a primitive
recursive setting will affect our proof strategies: as any primitive
recursive check converges, our constructions will be typically injury-free
and requirements will be solved one by one. That said, to build primitive
recursive objects, we won't need to worry about coding, and in fact we will
rely on a restricted form of Church-Turing thesis (see Remark \ref{rem:CTT}),
which will give to the paper more of a computability-theoretic flavour rather
than a complexi\-ty-the\-oretic one.

A final piece of motivation comes from the rapid emergence of \emph{online structure theory} (see, e.g., \cite{bazhenov2019foundations,DMN-survey,kalimullin2017algebraic,Mel-17,melnikov2020structure}), a subfield of computable structure theory
which deals with algorithmic situations in which unbounded search
is not allowed, that is formalized by focusing, e.g., on punctual presentations of structures, rather than computable ones.

\smallskip

The rest of the paper is organized as follows. In Section~\ref{sect:prelim},
we review background material. In Section~\ref{sect:normal-form}, we observe
a simple yet important fact: punctual equivalence relations have a normal
form. In the central sections of the paper, Sections~4-7, we prove that
$\mathbf{Peq}$ has several desirable features (which are unusual for degree
structures of equivalence relations): we show, in particular, that
$\mathbf{Peq}$ is a dense distributive lattice, in which each degree is
join-reducible and each degree below the top is meet-reducible. The last two
sections should convince the reader that, despite being surprisingly
well-behaved, $\mathbf{Peq}$ remains intricate: e.g., we prove that it
contains intervals which do not embed the diamond lattice; that it is neither
rigid nor homogeneous; and that it contains nonisomorphic lowercones. Several
questions remain open.

\section{Background, terminology, and notations} \label{sect:prelim}

We review some background material and terminology. For background on
computability theory, especially primitive recursive functions, the reader is
referred to \cite{odifreddi1992classical}.

\subsection{Equivalence relations}
As aforementioned, \emph{punctual} equivalence relations are primitive
recursive equivalence relations with domain $\omega$. Unless stated
otherwise, all our equivalence relations will be punctual. A
\emph{transversal} of an equivalence relation $R$ is any set $T$ such that $x
\cancel{\rel{R}} y$ for any distinct $x,y \in T$. The \emph{principal
transversal} $T_R$ of a given equivalence relation $R$ is the following
transversal,
\[
T_R:=\set{y \neq 0 : (\forall x < y)(x \cancel{R} y)}.
\]
Clearly, if $R$ is punctual, then $T_R$ is a primitive recursive set. As is
customary in theory of ceers, we denote by $\Id$ the identity on the natural
numbers. We often write $f: R \leq_{pr} S$ to mean that $f$ is a primitive
recursive reduction from $R$ to $S$.

\subsection{Finite equivalence relations}
We define an equivalence relation $R$ to be \emph{finite} if it has only
finitely many equivalence classes; $R$ is \emph{non-finite} otherwise. As in
the case of ceers~\cite{Gao} (but differently from the case of $\Delta^0_2$
equivalence relations~\cite{Bazhenov-et-al}), $\mathbf{Peq}$ has an initial
segment of order type $\omega$ consisting of the finite punctual equivalence
relations. More precisely, it is immediate to observe that
\[
\Id_1 <_{pr} \Id_2 <_{pr} \cdots <_{pr} \Id_{n} <_{pr}  \cdots,
\]
where $\Id_{n}$ is equality $\text{mod} \ n$.  Clearly each $\Id_{n}$ and
$\Id$ are punctual equivalence relations. In addition, any finite punctual
$R$ is $pr$-equivalent to $\Id_n$, for some $n \ge 1$, and $\Id_{n} \leq_{pr}
R$,  for every $n\ge 1$ and every non-finite punctual $R$.

\begin{remark}(Terminology)\label{rem:assumption}
In the rest of the paper, we shall assume that all our equivalence relations
are non-finite. Henceforth, by a \emph{punctual equivalence relation} we will
mean a primitive recursive relation with infinitely many equivalence classes.
Likewise, by a \emph{punctual degree} we will mean the $pr$-degree of a
punctual equivalence relation: of course, the equivalence relations lying in
a punctual degree are all punctual.
\end{remark}

\subsection{A listing of the primitive recursive
functions}\label{ssct:listing}

Throughout the paper we will refer to an effective listing $\{p_e\}_{e \in
\omega}$ of the primitive recursive functions, which can be found in many
textbooks: see e.g.~\cite{Hinman:Book} for a detailed definition of such a
listing. Let $T(e,x,z)$ and $U$ denote respectively Kleene's (primitive
recursive) predicate and a primitive recursive function such that for every
$e$, $\phi_{e}(x)=U(\mu z\,T(e,x,z))$ (where $\phi_{e}$ denotes the partial
computable function with index $e$ in the standard listing of the partial
computable functions), and let $p$ be a recursive function such that
$p_{e}=\phi_{p(e)}$: since $p$ comes from the $s$-$m$-$n$-theorem we may assume
that $p$ is primitive recursive. Let $V$ be the primitive recursive predicate
\[
V(e,x,y,s) \Leftrightarrow (\exists z < s)(T(p(e), x, z) \,\&\, U(z)=y):
\]
we will refer to $V(e,x,y,s)$ by saying that ``$p_e(x)$ has converged to $y$
in $< s$ steps'', and we will denote this by $p_e(x)[s]\!\downarrow =y$.
Similarly, it is primitive recursive to check whether ``$p_e(x)$ has
converged in $< s$ steps'' (denoted by $p_e(x) [s] \!\downarrow$), i.e.
$(\exists y < s)V(e,x,y,s)$.

\subsection{Binary strings}\label{ssct:strings}
We will use the standard notations and terminology about \emph{finite binary
strings}, which are the elements of the set $2^{<\omega}$. Let $\sigma, \tau$
be finite binary strings: we will denote by $l^{\sigma}$ the length of
$\sigma$; the concatenations of $\sigma$ and $\tau$ will be denoted by
$\sigma \widehat{\phantom{\alpha}} \tau$; if $i\in \{0,1\}$ is a number then
$\str{i}$ denotes the string of length $1$, consisting of the single bit $i$;
the symbol $\lambda$ denotes the empty string.

We will freely identify a set $X \subseteq \omega$ with its characteristic
function, thus viewing $X$ as a member of $2^{\omega}$. If $k \in \omega$ and
$X$ is a set, then the symbol $X\restriction  k$ denotes the initial segment
of $X$ (thus a member of $2^{<\omega}$) of length $k$. Given $\sigma \in
2^{<\omega}$ and a set $X$, let
\[
\sigma \ast X=
\begin{cases}
\sigma(i), &\text{if $i< l^{\sigma}$}\\
X(i-l^\sigma), &\text{otherwise}.
\end{cases}
\]

In analogy with the common usage for sets where $\overline{X}=\{x: X(x)=0\}$,
given a string $\sigma \in 2^{<\omega}$ we denote $\overline{\sigma}=\{i: i<
l^{\sigma} \ \&\ \sigma(i)=0\}$, and we let $\ZeS=\#(\overline{\sigma})$, i.e. the
number of $0$'s in the range of $\sigma$.

We will assume a ``primitive recursive coding'' of the binary strings, with
primitive recursive length function, projections, etc.

\begin{remark}\label{rem:CTT}
Throughout this paper we will often build primitive recursive functions or
primitive recursive sets. It is sometimes convenient to use the following
analogue of the Church-Turing thesis for primitive recursive functions:

Let $\varphi_e$ be the $e^{th}$ partial (general) computable function. Then
(in accordance with what stated in Section~\ref{ssct:listing}) every
primitive recursive function is equal to some $\varphi_e$ where the
computation for $\varphi_e$ runs in primitive recursively many steps. Thus,
to define a primitive recursive function $f$ (or a set), it is enough to
specify a general algorithm for computing $f(n)$ as long as the number of
steps taken to decide $f(n)$ is bounded by a primitive recursive function.

This helps to avoid overly formal and cumbersome definitions, since it is often easier to see that the time bound is primitive recursive.
\end{remark}

\section{The normal form theorem and some structural properties} \label{sect:normal-form}

In the literature about computable reducibility, it is common to use the
following way of encoding  sets of numbers by equivalence relations: given
$X\subseteq \omega$, one defines
\[
x R_X y \Leftrightarrow x,y\in X \text{ or $x=y$}.
\]
Equivalence relations of the form $R_X$ are called \emph{1-dimensional} in
\cite{Gao},  while $c$-degrees containing $1$-dimensional equivalence
relations are called \emph{set-induced} in \cite{Ng}. The interesting feature of
set-induced degrees is that they offer algebraic and logical information
about the overall structure of $c$-degrees: for example, in
\cite{Andrews-et-al-14} it is proved that the first-order theory of ceers is
undecidable, by showing that the interval $[\deg_c(\Id), \deg_c(R_K)]$ of the
$c$-degrees is isomorphic to the interval $[\mathbf{0}_1, \mathbf{0}'_1]$ of
the $1$-degrees, where $\mathbf{0}_1$ is the $1$-degree of an infinite and
co-infinite computable set and $\mathbf{0}'_1$ is the $1$-degree of the
halting set $K$.

Yet, set-induced degrees are far from exhausting the collection of all
$c$-degrees. In fact, if $R$ is an equivalence relation with  two
non-computable $R$-classes, then there is obviously no $X$ such that
$R\equiv_c R_X$. When dealing with punctual equivalence relations, the
situation changes: all $pr$-degrees are set-induced. Specifically, the next easy theorem shows that the punctual complexity of $R$ is
entirely encoded in its principal transversal $T_R$.

\begin{theorem}[Normal Form Theorem]\label{thm:normal-form}
For any punctual $R$, we have that
\[
R\equiv_{pr} R_{\overline{T_R}}.
\]
\end{theorem}

\begin{proof}
It is immediate to see that the function  $f(x) = (\mu y\leq x) [y\rel{R}x]$ is
primitive recursive and reduces $R$ to  $R_{\overline{T_{R}}}$. On the other
hand, the following primitive recursive function $g$ reduces
$R_{\overline{T_R}}$ to $R$,
\[
g(x)=
\begin{cases}
0, &\text{if $x\in \overline{T_{R},}$}\\
x, &\text{otherwise}.
\end{cases}
\]
Hence, $R\equiv_{pr} R_{\overline{T_R}}$.
\end{proof}

So, whenever is needed, one can assume that a given punctual equivalence
relation is in normal form. Furthermore, the next lemma ensures that, if two
punctual equivalence relations $R$ and $S$ are in normal form and $R\leq_{pr}
S$, then there is a reduction from $R$ to $S$ that  ``respects'' the normal
form.

\begin{lemma}\label{lem:from-classone-to-classone}
If $R_X \leq_{pr} R_Y$, then there is a reduction $g$ from $R_X$ to $R_Y$ such that
$g[X] \subseteq Y$.
\end{lemma}

\begin{proof}
Let $f$ be a primitive recursive function reducing $R_X$ to $R_Y$ such that
$f[X] =\{a\}$ with $a \notin Y$. Fix $y \in Y$ and define
\[
g(x)=
\begin{cases}
y, &\text{if $x \in X$},\\
a, &\text{if $x \notin X$ and $f(x)\in Y$},\\
f(x), &\text{if $x \notin X$ and $f(x)\notin Y$}.
\end{cases}
\]
It is easy to see that $g$ is primitive recursive, maps $X$ to $Y$ and
reduces $R_X$ to $R_Y$.
\end{proof}

Another assumption that one can make without loss of generality is that a
$pr$-reduction between punctual equivalence relations is surjective on the
equivalence classes of its target. This contrasts with the case of ceers
where (see \cite{Andrews-Sorbi-19}) if $R,S$ are such that $R \leq_c S$ via a
computable function $f$ whose range hits all the equivalence classes of $S$,
then the reduction can be inverted, i.e., $S \leq_{c} R$ as well. This
inversion lemma fails for punctual equivalence relations, in fact:

\begin{lemma}\label{lem:inversionfail}
If $R\leq_{pr} S$ then there exists $g$ such that $g: R\leq_{pr} S$ and $g$
hits all the $S$-classes.
\end{lemma}

\begin{proof}
Let $R\leq_{pr} S$ via $f$: by Theorem~\ref{thm:normal-form} we may assume
that $R=R_X$ and $S=R_Y$ for some pair of primitive recursive sets and by
Lemma~\ref{lem:from-classone-to-classone} we may assume that $f$ maps $X$ to
$Y$. Define
\[
g(x)=
\begin{cases}
f(x), &\text{if $x \in X$},\\
\mu y\,[y \leq \max\{f(i): i\leq x\} \,\&\, y \in \overline{Y}
               \smallsetminus \{g(i) : i < x\}],
&\text{if $x \notin X$}.
\end{cases}
\]
Then $g$ gives the reduction and is onto $\overline{Y}$: to show primitive
recursiveness of $g$ we use the fact that if $x$ is the $n$-th element in
$\overline{X}$, then the $n$-th element of $\overline{Y}$ in order of
magnitude is $\leq \max\{f(i): i \leq x\}$ by injectivity of $f$ on
$\overline{X}$.
\end{proof}

\begin{remark}\label{rem:nondecreasing}
Note that the function $g: R_{X} \leq_{pr} R_Y$ constructed in the last lemma
is nondecreasing on $\overline{X}$, i.e., if $x<y$ and $x,y\in \overline{X}$
then $g(x)<g(y)$.
\end{remark}

We may comment on the results presented so far by saying that investigating
punctual equivalence  relations under $\leq_{pr}$ turns out to be the same as
investigating primitive recursive sets under a primitive recursive
reducibility that is required to be bijective on the complements of the sets.
In the rest of the paper, we will take advantage of this correspondence by
often constructing, instead of a full punctual equivalence relation $R$, only
a primitive recursive set $Y$ corresponding to its \emph{main class} in
normal form (i.e.\ $R \equiv_{pr} R_Y$).

\section{Incomparability of punctual degrees}
In this and the following section, we tackle some of the most natural
questions that one can formulate about a new degree structure.

\subsection{The greatest  punctual degree}

We begin by proving that there is a greatest punctual degree.
\begin{proposition}
$\mathbf{Peq}$ has a greatest element. In fact, for all punctual $R$,
$R\leq_{pr} \Id$.
\end{proposition}

\begin{proof}
Given $R$,  let $f$ be the primitive recursive function from the proof of the
Normal Form Theorem (i.e., the function reducing $R$ to
$R_{\overline{T_R}})$. Note that $f$ is also a reduction from $R$ to $\Id$,
as it maps equivalence classes to singletons.
\end{proof}

The punctual degree of $\Id$ contains many natural examples of equivalence
relations. For example, in the literature about polynomial-time reducibility
(see, e.g., \cite{buss2011strong}) researchers considered the isomorphism relations of
familiar classes of \emph{finite} structures, such as graphs, groups, trees,
linear orders, Boolean algebras, and so forth. It is not difficult to see
that all these relations turn out to be $pr$-equivalent to $\Id$. Consider
for instance ${\mathrm{GI}}$, the isomorphism relation between finite graphs.  On
one hand, the problem of deciding whether two finite graphs are isomorphic is
primitive recursive (in fact, it belongs to $\mathbf{NP}$). On the other hand, a $pr$-reduction from $\Id$ to ${\mathrm{GI}}$ is readily obtained by assigning, to each $n$, the empty graph on the domain $\{i: i<n\}$. Hence, $\mathrm{GI}$ is $pr$-equivalent to $\Id$.

Anyway, the fact that the punctual degree of $\Id$ is large does not come as a
surprise. In fact, to construct a computable function which is not primitive
recursive requires non-trivial work (recall Ackermann's famous
construction~\cite{Ackermann}). And we show next that a punctual equivalence
relation $R$  lies strictly below $\Id$ only if, when presented in normal
form, the set of its singletons cannot be enumerated in a primitive recursive
way without repetitions.

To be more precise, let us introduce first the following  analogue of
immunity for primitive recursive sets.

\begin{emdef}\label{def:presets}
A set $X\subseteq \omega$ is \emph{primitive recursively enumerable}
(abbreviated by \emph{p.r.e.}) if $X$ is the range of an injective primitive
recursive function. $X$ is \emph{primitive recursively immune} (abbreviated
by \emph{p.r.-immune}) if $X$ is infinite and it has no infinite p.r.e.\
subset.
\end{emdef}

\begin{remark}
Note that we define a set as p.r.e.\ only if it  has a primitive recursive
enumeration which is \emph{injective}. Without injectivity one would obtain
all c.e.\ sets: it follows easily from Kleene's Normal Form Theorem that any
c.e.\ set can be enumerated by a primitive recursive function. In contrast,
Theorem~\ref{prop:belowid} shows that the p.r.e.\ sets (as just defined) form
a proper subclass of the c.e.\ sets, and in fact even of the primitive
recursive sets.
\end{remark}

\begin{proposition} \label{prop:topcrit}
If $X$ is any set of numbers then $\Id\not\leq_{pr} R_{X}$ if and only if
$\overline{X}$ is p.r.-immune.
\end{proposition}

\begin{proof}
$(\Rightarrow)$: If $\overline{X}$ contains a p.r.e.\ set $A$, then $\Id
\leq_{pr} R_X$ via any primitive recursive function that injectively lists
$A$.

$(\Leftarrow)$:  If $\Id \leq_{pr} R_{X}$ via some primitive
recursive function $f$,  then, with the exception of at most one element, the
range of $f$ is contained in $\overline{X}$.  Therefore, $\range(f)\cap
\overline{X}$ is an infinite p.r.e.\ set showing that $\overline{X}$ is not
p.r.-immune.
\end{proof}

\begin{theorem}\label{prop:belowid}
There exists a primitive recursive set $Y$ whose complement is p.r.-immune.
\end{theorem}

\begin{proof}
We construct $Y$ in stages by approximating its characteristic function,
i.e., $Y=\bigcup_{s\in\omega}\sigma_s$, where $l^{\sigma_s}=s$. The
construction of $Y$ is similar to that of a simple set and is rather
straightforward. We describe it in some detail to let the reader familiarize themselves
with the sort of machinery that we employ in more intricate
constructions.

\smallskip

We aim at satisfying the following requirements:
 \begin{align*}
  &P_{e}: \text{if $p_e$ is injective, then $\range(p_e)\cap Y\neq \emptyset$},\\
  &M: Y \text{ is co-infinite},
\end{align*}
where  $\set{p_e}_{e\in\omega}$ is a computable list of all primitive
recursive functions, as in Section~\ref{ssct:listing}.

The strategy for a $P_e$-requirement works as follows. During the so-called
``$P_e$-cycle'' we enlarge the set $Y$ (by setting $\sigma_{s+1} =
\catt{\sigma_s}{1}$ at the current stage $s$) until we see that one of the
following conditions holds: either the function $p_e$ is not injective, or we
have $p_e(z) \in Y$ for some $z$. Let $s_0+1$ be the first stage at which we
witness such a situation. We set $\sigma_{s_0+1} = \catt{\sigma_{s_0}}{0}$,
close the $P_e$-cycle and move to satisfying the $P_{e+1}$-strategy, by
opening the $P_{e+1}$-cycle.

\subsubsection*{The construction}
At the beginning of any nonzero stage of the construction we assume that
there exists exactly one open $P$-cycle. Section~\ref{ssct:listing} will
guarantee that the various checks involving $p_e$-computations and their
convergence are primitive recursive.

\subsubsection*{Stage $0$}
Define $\sigma_0 = \lambda$; \emph{open the $P_0$-cycle}. Thus at the beginning
of the next stage, there will be exactly one open $P$-cycle, namely the
$P_{0}$-cycle.

\subsubsection*{Stage $s+1$}
Assume that the $P_e$-cycle is the currently open cycle. We distinguish three
cases:
\begin{enumerate}
	\item There are $l,m \leq s$ such that $p_e(l) [s] \!\downarrow\ =
p_e(m) [s] \!\downarrow$: if so, \emph{close the $P_e$-cycle}, define
$\sigma_{s+1}=\catt{\sigma_s}{0}$, and \emph{open the $P_{e+1}$-cycle}.
\item There is $m\leq s$ such that $p_e(m) [s] \!\downarrow\ =z$ and
    $\sigma_s(z)\!\downarrow\ =1$: do the same as in $(1)$.
\item Otherwise: keep the $P_e$-cycle open and define $\sigma_{s+1}=
    \catt{\sigma_s}{1}$.
\end{enumerate}
Again it is immediate to see that the next stage will inherit from this stage
exactly one open $P$-cycle.

\subsubsection*{The verification} The verification relies on the following lemmas.

\begin{lemma}
Every $P_{e}$-cycle is eventually opened and is later closed forever.
\end{lemma}
\begin{proof}
Notice that the $P_{e}$-cycle is opened at stage $s$ if and only if $e=0$ and
$s=0$, or $e>0$ and we close at $s$ the $P_{e-1}$-cycle.

So assume by induction that the lemma is  true of every $i<e$: thus there
exists a unique $s_{0}$ at which we open the $P_{e}$-cycle ($s_0=0$ if $e=0$,
or otherwise $s_0$ is the stage at which we close the $P_{e-1}$-cycle). If
the $P_{e}$-cycle does not satisfy the claim then the construction implies
the following: the function $p_e$ is injective and $\sigma_{s+1}=
\catt{\sigma_s}{1}$ for all $s\geq s_0+1$. Therefore since $p_e$ is injective
and $Y$ is cofinite, there would be infinitely many elements $m$ with
$p_e(m)\in Y$. Thus, at some stage $s_1+1 \geq s_0+1$, the $P_e$-cycle would
be closed by the item (2) of the construction, and never opened again. We
obtain a contradiction, showing that the $P_{e}$-cycle satisfies the claim.
\end{proof}

\begin{lemma}
$Y$ is primitive recursive and co-infinite.
\end{lemma}
\begin{proof}
It is enough to observe that for all $s$, the value of $\sigma_{s}$ is
decided at stage $s$, and therefore the function $s\mapsto \sigma_s$ is
primitive recursive. Moreover, it is immediate to check that
$l^{\sigma_s}=s$, for every $s$. Hence, $Y=\bigcup_{s\in \omega} \sigma_s$ is
primitive recursive, as $Y(s)=\sigma_{s+1}(s)$.

Since every $P_e$-cycle eventually closes, there are infinitely many stages
$s$ with $\sigma_{s+1} = \catt{\sigma_s}{0}$. This implies that the set
$\overline{Y}$ is infinite.
\end{proof}

\begin{lemma}
All $P$-requirements are satisfied.
\end{lemma}

\begin{proof}
Suppose that $p_e$ is an injective function. Consider the stage $s_0$ at
which the $P_e$-cycle closes. Clearly, the closure of the $P_e$-cycle is
triggered by item~(2). Thus, there is a number $m\leq s_0$ with
$\sigma_{s_0}(p_e(m))\!\downarrow\ = 1$. Hence, we have $p_e(m) \in Y$ and
$\range(p_e) \cap Y \neq \emptyset$.
\end{proof}

Theorem~\ref{prop:belowid} is proved.
\end{proof}

Combining the last three propositions we immediately obtain that there exists
a punctual degree strictly below the identity.

\begin{corollary}
There exists a punctual $R$ such that $R <_{pr} \Id$.
\end{corollary}

\subsection{Counterexamples to reducibilities}\label{ssct:counterexample}

In the rest of the paper, we will often need to build a punctual $S$ such
that, for a given $R$, we have $R \nleq_{pr} S$. To do so, we construct a primitive
recursive set $Y$ such that for every $e$, the requirement
\[
P_e:  p_e \mbox{ does not reduce $R$ to $R_{Y}$},
\]
is satisfied, and thus $S=R_{Y}$ is our desired equivalence relation.
Typically, we construct an increasing sequence $\{\sigma_{s}: s \in \omega\}$
of strings in $2^{<\omega}$ so that $Y=\bigcup_{s} \sigma_{s}$.

At a stage
$s+1$ of the construction we say that  \emph{$p_{e}$ shows a counterexample
to $R\leq_{pr} R_Y$} if there exist $l,m \leq s$ so that
$p_{e}(l)[s]\downarrow$, $p_{e}(m)[s]\downarrow$, and $p_{e}(l)[s], p_{e}(m)[s] <
l^{\sigma_{s}}$, and
\[
l \rel{R} m \Leftrightarrow (\sigma_{s}(p_{e}(l))\ne
\sigma_{s}(p_{e}(m)))  \text{ or } (p_e(l)\neq p_e(m) \,\&\,
\sigma_{s}(p_{e}(l)) = \sigma_{s}(p_{e}(m)) = 0).
\]
A counterexample will guarantee, for the final $Y$, that
\[
l \rel{R} m \Leftrightarrow p_{e}(l) \, \cancel{\rel{R_Y}}\, p_{e}(m),
\]
as desired. For a given $\sigma_{s}$, primitive recursiveness of $R$ and
Section~\ref{ssct:listing} will guarantee that checking if a counterexample
is being shown is primitive recursive.

Similarly, if for every $e$ we want to satisfy the requirement
\[
Q_e:  p_e \mbox{ is not a reduction from $R_{Y}$ to $R$},
\]
at a stage $s+1$ of the construction we say that  \emph{$p_{e}$ shows a
counterexample to $R_Y\leq_{pr} R$} if there exist $l,m < l^{\sigma_{s}}$ so
that $l\neq m$, $p_{e}(l)[s]\downarrow$, $p_{e}(m)[s]\downarrow$, and
\[
\sigma_{s}(l)=\sigma_{s}(m)=1 \Leftrightarrow p_{e}(l) \,\cancel{\rel{R}}\, p_{e}(m):
\]
a counterexample guarantees for the final $Y$ that $R_Y \nleq_{pr} R$.

Finally, sometimes we build two primitive recursive sets $X,Y$ in the form
$X=\bigcup_{s} \sigma^X_{s}$, and $Y=\bigcup_{s} \sigma^Y_{s}$. At a stage
$s+1$ of the construction we say that  \emph{$p_{e}$ shows a counterexample
to $R_X\leq_{pr} R_Y$} if there exist $l,m < l^{\sigma^X_{s}}$ so that $l\neq m$,
$p_{e}(l)[s]\downarrow$, $p_{e}(m)[s]\downarrow$, and $p_{e}(l)[s], p_{e}(m)[s] <
l^{\sigma^Y_{s}}$, and
\begin{multline*}
\sigma^X_{s}(l)= \sigma^X_{s}(m)=1 \Leftrightarrow (\sigma^Y_{s}(p_{e}(l))
\ne \sigma^Y_{s}(p_{e}(m))) \text{ or }\\
(p_e(l)\neq p_e(m) \,\&\, \sigma^Y_{s}(p_{e}(l)) = \sigma^Y_{s}(p_{e}(m)) = 0).
\end{multline*}

\subsection{Incomparability}\label{ssct:incomparability}

Among non-finite punctual equivalence relations, there is no least punctual
degree. In fact, $\mathrm{Id}$ is the only  punctual equivalence relation which is
non-finite and $pr$-comparable with all other punctual equivalence
relations.

The next result will be a consequence also of Theorem~\ref{thm:density} and
Theorem~\ref{thm:density-plus-incomparability} (to be discussed later). However, it may be
useful to give a direct proof in order to introduce the
incomparability strategy, which will be exploited in other constructions.

\begin{theorem}\label{thm:incomparability}
For any punctual $R<_{pr} \Id$, there is punctual $S$ such that $S|_{pr} R$.
\end{theorem}

\begin{proof}
Given $R<_{pr}\Id$, we build by stages a primitive recursive set $Y$ such
that $S=R_Y$ satisfies the claim. The requirements to be satisfied are
\begin{align*}
P_e: p_e \mbox{ is not a reduction from $R$ to $R_Y$},\\
Q_e: p_e \mbox{ is not a reduction from $R_Y$ to $R$}.
\end{align*}

\subsubsection*{The strategy}
To satisfy $P_e$, one continues putting more and more fresh elements into $Y$, thus not increasing
the number of $R_Y$-classes. By doing so, we will witness eventually that
$p_e$ maps two $R$-classes to a single $R_Y$-class, since the number of
$R$-classes will outgrow the number of $R_Y$-classes (recall that $R$ is
non-finite, see Remark~\ref{rem:assumption}).

To satisfy a given $Q_e$-requirement, we follow a dual strategy: for any fresh element,
we declare the corresponding singleton as an $R_Y$-class. This ensures that we will find
a pair of witnesses that show that $p_e$ is not a reduction of $R_Y$ to $R$,
since otherwise we would have a reduction of $\Id$ to $R$ as well.

\subsubsection*{The construction}
We construct set $Y$ in stages: in fact at a stage $s$ we define its initial
segment $\sigma_s$ of length $s$, and eventually we take
$Y=\bigcup_{s\in\omega}\sigma_s$.

\subsubsection*{Stage $0$}
Define $\sigma_{0}=\lambda$; \emph{open the $P_{0}$-cycle}.

\subsubsection*{Stage $s+1$}
Assume that $R$ is the currently open cycle.
\begin{enumerate}
\item $R=P_e$, for some $e$: If so, we distinguish two cases.

\begin{enumerate}
\item  If $p_{e}$ shows a counterexample to $R\leq_{pr} R_Y$ (see
    Section~\ref{ssct:counterexample}) then \emph{close the
    $P_{e}$-cycle}. Define $\sigma_{s+1}=\catt{\sigma_{s}}{0}$.
    \emph{Open the $Q_{e}$-cycle }.
\item Otherwise, keep the $P_{e}$-cycle open and define
    $\sigma_{s+1}=\catt{\sigma_{s}}{1}$.
\end{enumerate}

\smallskip

\item If $R=Q_e$, for some $e$, then distinguish two more cases.

\begin{enumerate}
\item  If  $p_{e}$ shows a counterexample to $R_Y\leq_{pr} R$ (see
    Section~\ref{ssct:counterexample}) then \emph{close the
    $Q_{e}$-cycle}. Define $\sigma_{s+1}=\catt{\sigma_{s}}{0}$.
    \emph{Open the $P_{e+1}$-cycle}.
\item Otherwise, keep the $Q_{e}$-cycle open and define
    $\sigma_{s+1}=\catt{\sigma_{s}}{0}$.
\end{enumerate}
\end{enumerate}

This concludes the construction.

\subsubsection*{The verification}
 The verification relies on the following lemmas.

\begin{lemma}
$Y$ is primitive recursive.
\end{lemma}
\begin{proof}
The function $s\mapsto \sigma_s$ is primitive recursive; moreover
$l^{\sigma_{s}}=s$. Hence, $Y$ is primitive recursive, as
$Y(s)=\sigma_{s+1}(s)$.
\end{proof}

\begin{lemma}
All $P$-requirements are satisfied.
\end{lemma}

\begin{proof}
The proof follows the lines of the analogous claim in the proof of
Theorem~\ref{prop:belowid}. First of all, it easily follows by induction that
the $P_{e}$-cycle is opened at stage $0$ if $e=0$, and at the stage at which
the $Q_{e-1}$-cycle is closed if $e>0$; and the $Q_{e}$-cycle is opened at
the stage at which the $P_{e}$-cycle is closed. Assume that for every $i<e$
the $P_{i}$-cycle and the $Q_{i}$-cycle have been closed, and the
corresponding requirements are satisfied. Then at the stage $s_{0}$ (with
$s_{0}=0$ if $e=0$, or $s_{0}$ is the stage when we close the
$Q_{e-1}$-cycle) we open the $P_{e}$-cycle. Failure to close the $P_e$-cycle
would entail that $Y$ is cofinite, thus $R_{Y}$ would have only finitely many
classes, but $p_{e}$ never showing a counterexample would give that $p_{e}: R
\leq_{pr} R_{Y}$, a contradiction as $R$ is not finite. Thus, at some stage,
$p_{e}$ shows a counterexample to $R \leq_{pr} R_{Y}$, whence $P_{e}$ is
satisfied.
\end{proof}

\begin{lemma}
All $Q$-requirements are satisfied.
\end{lemma}
\begin{proof}
Assume that for every $i<e$ the $Q_{i}$-cycle has been closed, and
for every $i\leq e$ the $P_{i}$-cycle has been closed, and the corresponding
requirements are satisfied. If $s_{0}$ is the stage at which we open the
$Q_{e}$-cycle (that is when we close the $P_{e}$-cycle) and for every $s\ge
s_{0}+1$ we never close the cycle then this would entail that $Y$ is finite
(whence $R_{Y}\equiv_{pr} \Id$), but as $p_{e}$ never shows a counterexample
this would give that $p_{e}: R_{Y} \leq_{pr} R$, whence $\Id \leq_{pr} R$, a
contradiction. Thus, at some stage, $p_{e}$ shows a counterexample to $R_{Y}
\leq_{pr} R$, whence $Q_{e}$ is satisfied.
\end{proof}

This concludes the proof of Theorem \ref{thm:incomparability}.
\end{proof}

From the last theorem, it follows that there is an infinite antichain of
punctual degrees.

\begin{corollary}\label{cor:antichain}
There are punctual equivalence relations $\set{S_i}_{i\in\omega}$ such that
$S_i |_{pr}S_j$ for all $i\neq j$.
\end{corollary}

\begin{proof}
The proof relies on the following observation

\medskip

\paragraph{\textbf{Observation}}
Suppose that $\{R_{i}: i \in \omega\}$ is a family of punctual equivalence
relations none of which is $pr$-equivalent to $\Id$, and for which there
exists a primitive recursive predicate $U(i,x,y)$ such that $x \rel{R_i} y$
if and only if $U(i,x,y)$. Then a straightforward modification of the proof
of Theorem~\ref{thm:incomparability} will show that there exists a punctual
$S$ such that $R_{i} |_{pr} S$, for every $i$. The construction in  this case
aims to build a primitive recursive set $Y$ satisfying the following
requirements $R_{\langle e,i\rangle}$ indexed by the values of the primitive
recursive Cantor pairing function:
\begin{align*}
P_{\langle e, i \rangle}: p_e \mbox{ does not reduce
          $R_{i}$ to $R_Y$},\\
Q_{\langle e, i \rangle}: p_e \mbox{ does not reduce $R_Y$ to $R_{i}$}.
\end{align*}
The construction goes by opening and closing $R$-cycles as in the proof of
the previous theorem. Checking if a counterexample is being shown is
primitive recursive, since this can be done by using the primitive recursive
predicate $U$, together with Section~\ref{ssct:listing}.

\medskip

To finish the proof of the corollary, define by induction the following
infinite antichain $\{S_{n}\}_{n \in \omega}$ of punctual equivalence
relations: Pick $S_{0}<_{pr} \Id$; having found $S_{0}, \ldots, S_{n}$ apply
the above observation to the family $\{R_{i}\}_{i \in \omega}$, where
$R_{i}=S_{i}$ if $i < n$, and $R_{i}=S_{n}$ if $i \ge n$.
\end{proof}

\section{The punctual degrees form a distributive lattice}

In this section, we study the structure of the punctual degrees under joins
and meets. By the Normal Form Theorem we will confine ourselves to
equivalence relations of the form $R_{X}$, where $X$ is a primitive recursive
set.

The first result of this section provides us with a useful characterization
of the reducibility $\leq_{pr}$. Informally speaking, the characterization
connects the punctual degree of a relation $R_X$ with the growth rate of the
function
\[
	\#(0^X) [s] := \#(0^{X\restriction (s+1)}),
\]
i.e. the function which counts the number of singleton $R_X$-classes.

We emphasize that for a primitive recursive $X$, the corresponding function
$\#(0^X) [s]$ is also primitive recursive.

\begin{proposition} \label{prop:growth-rate}
Let $X$ and $Y$ be coinfinite primitive recursive sets. Then we have
$R_X\leq_{pr} R_{Y}$ if and only if there exists a primitive recursive
function $h(x)$ such that for all $s$,
\[
	\#(0^X) [s] \leq \#(0^Y) [h(s)].
\]
\end{proposition}

\begin{proof}
Suppose that $f\colon R_X \leq_{pr} R_Y$. By Remark~\ref{rem:nondecreasing},
one may assume that for any pair of elements $k<l$ from $\overline{X}$, we
have $f(k) < f(l)$. In addition, by
Lemma~\ref{lem:from-classone-to-classone}, we assume that $f[\overline{X}]
\subseteq \overline{Y}$. The desired primitive recursive function $h(s)$ is
defined as follows:
\[
	h(s) = \begin{cases}
		f(k^{\ast}), & \text{if } k^{\ast} \text{ is the greatest number}\\
		& \text{ such that }  k^{\ast} \leq s \text{ and } k^{\ast} \in \overline{X},\\
		0, & \text{if } (\forall k \leq s) (k\in X).
	\end{cases}
\]
Indeed, suppose that $\#(0^X) [s] = N>0$. Consider the set
\[
	\overline{X} \cap \{0,1,\dots,s\} = \{ k_1 < k_2 < \dots < k_N\}.
\]
Then we have $f(k_1) < f(k_2) < \dots < f(k_N) = h(s)$, and hence, $\#(0^Y)
[h(s)] \geq N = \#(0^X) [s] $.

To show the converse, let $h$ be a primitive recursive function such that
$\#(0^X) [s] \leq \#(0^Y) [h(s)]$ for all $s$. Without loss of generality,
one may assume that $0\in Y$. The desired primitive recursive reduction
$g\colon R_X \leq_{pr} R_{Y}$ is defined by recursion on $x\in\omega$ as
follows:
\begin{align*}
	g(0)	&= \begin{cases}
		0, & \text{if } 0 \in X,\\
		\text{the least element from } \overline{Y}, & \text{if } 0 \in\overline{X};
	\end{cases} \\
	g(x+1) &= \begin{cases}
		0, & \text{if } x+1 \in X,\\
		\mu y[ y\leq h(x+1) \,\&\, y\in \overline{Y} \smallsetminus \{ g(z)\,\colon z\leq x\} ], & \text{if } x+1 \in\overline{X}.
	\end{cases}
\end{align*}
Suppose that $x\in \overline{X} \smallsetminus\{ 0\}$ and the set
$\overline{X}\cap \{ 0,1,\dots,x\}$ equals $\{ k_1 < k_2 <\dots < k_N = x\}$.
Then the set $\overline{Y} \cap \{1,2,\dots,h(x)\}$ contains at least $N$
elements, and this set has at most $N-1$ elements from
$\text{range}(g\upharpoonright x)$. Therefore, the function $g$ is
well-defined. In addition, it is easy to observe that $g$ provides a
reduction $R_X \leq_{pr} R_{Y}$.
\end{proof}

Proposition~\ref{prop:topcrit} can now be restated as:

\begin{corollary}
	$\Id \leq_{pr} R_Y$ if and only if there is a primitive recursive function $h(x)$ such that $s \leq \#(0^Y) [h(s)]$ for all $s\in\omega$.
\end{corollary}

\subsection{Joins and meets}
Now we are ready to prove that the partial order $\mathbf{Peq}$ has joins and
meets, which make the structure a lattice. By slightly abusing notations, we
will talk about suprema and infima of punctual equivalence relations
(referring of course to the poset $\mathbf{Peq}$ of the $pr$-degrees).

\begin{theorem}\label{theor:lattice}
The structure $\mathbf{Peq}$ is a lattice.
\end{theorem}

\begin{proof}
Suppose that $X$ and $Y$ are coinfinite primitive recursive sets such that
$R_X |_{pr} R_Y$. Without loss of generality, we assume that $0\in X\cap Y$.
In order to prove the theorem, it is sufficient to show that the relations
$R_X$ and $R_{Y}$ have supremum and infimum.

We define a set $Z_0 \subseteq \omega$ as follows: $0\in Z_0$, and
\[
	s+1 \not\in Z_0\ \Leftrightarrow\ \max(\#(0^X) [s+1],
\#(0^Y) [s+1]) > \max (\#(0^X) [s],  \#(0^Y) [s]).
\]
Recall that the functions $\#(0^X)[s]$ and $\#(0^Y)[s]$ are primitive
recursive. Hence, it is easy to show that the set $Z_0$ is primitive
recursive and coinfinite.

\begin{claim}
$R_{Z_0}$ is the supremum of $R_X$ and $R_Y$.
\end{claim}
\begin{proof}
	First, we note the following:	$\#(0^Z)[0]  = \#(0^X)[0] = \#(0^Y)[0] = 0$, and every set $U$ satisfies
	\[
		\#(0^U)[s+1] = \begin{cases}
			\#(0^U)[s] + 1, & \text{if } s+1\not\in U,\\
			\#(0^U)[s],  & \text{if } s+1\in U.
		\end{cases}
	\]
	These observations (together with an easy induction argument) imply that
	\begin{equation} \label{equ:maximum}
		\#(0^{Z_0})[s] = \max(\#(0^X) [s],  \#(0^Y) [s]).
	\end{equation}
	Thus, one can apply Proposition~\ref{prop:growth-rate} for the function $h(x)=x$, and deduce that $R_{Z_0}$ is an upper bound for both $R_X$ and $R_Y$.
	
	Now suppose that $R_{V}$ is an arbitrary upper bound of $R_X$ and $R_Y$. By Proposition~\ref{prop:growth-rate}, we choose primitive recursive functions $h_X$ and $h_Y$ such that $\#(0^X) [s] \leq \#(0^V) [h_X(s)]$ and $\#(0^Y) [s] \leq \#(0^V) [h_Y(s)]$. Then by~(\ref{equ:maximum}), we obtain
	\[
		\#(0^{Z_0})[s] \leq \max(\#(0^V) [h_X(s)], \#(0^V) [h_Y(s)]) = \#(0^V) [\max(h_X(s),h_Y(s)) ].
	\]
	Hence, we deduce that $R_{Z_0} \leq_{pr} R_V$, and $R_{Z_0}$ is join of $R_X$ and $R_Y$.
\end{proof}

Let now $Z_1$ be the set determined by the following: $0\in Z_1$, and
\[
	s+1 \not\in Z_1\ \Leftrightarrow\ \min(\#(0^X) [s+1],  \#(0^Y) [s+1]) > \min (\#(0^X) [s],  \#(0^Y) [s]).
\]

\begin{claim}
$R_{Z_1}$ is the infimum of $R_X$ and $R_Y$.
\end{claim}
\begin{proof}
	As in the previous claim, one can easily show that
	\begin{equation} \label{equ:minimum}
		\#(0^{Z_1})[s] = \min(\#(0^X) [s],  \#(0^Y) [s]).
	\end{equation}
	By Proposition~\ref{prop:growth-rate}, $R_{Z_1}$ is a lower bound for $R_X$ and $R_Y$.

	Let $R_{V}$ be a lower bound of $R_X$ and $R_Y$. We fix primitive recursive functions $q_X$ and $q_Y$ such that $\#(0^V) [s] \leq \#(0^X) [q_X(s)]$ and $\#(0^V) [s] \leq \#(0^Y) [q_Y(s)]$. Then by~(\ref{equ:minimum}),
	\begin{multline*}
		\#(0^V) [s] \leq \min( \#(0^X) [q_X(s)],  \#(0^Y) [q_Y(s)]) \leq\\
		\min(\#(0^X) [\max(q_X(s),q_Y(s))],  \#(0^Y) [\max(q_X(s),q_Y(s))]) =\\
		 \#(0^{Z_1}) [\max(q_X(s),q_Y(s))].
	\end{multline*}
	Therefore, $R_{Z_1}$ is meet of $R_X$ and $R_Y$.
\end{proof}

Theorem~\ref{theor:lattice} is proved.
\end{proof}

\begin{emdef}\label{def:joinsets}
Given primitive recursive sets $X$ and $Y$, let us denote $R_X \lor R_Y$ the
relation $R_{Z_0}$, constructed in the proof of Theorem~\ref{theor:lattice},
giving the supremum of $R_X$ and $R_Y$. In addition, let us denote $Z_0=X\lor
Y$. Likewise, let us denote $R_X \wedge R_Y$ the relation $R_{Z_1}$,
constructed in the proof above, giving the infimum of $R_X$ and $R_Y$. We
also denote $Z_1=X\wedge Y$.
\end{emdef}

\begin{corollary}
There are no minimal pairs of punctual degrees.
\end{corollary}

\begin{proof}
Immediate.
\end{proof}

Note that in the proof of Theorem~\ref{theor:lattice}, we gave an explicit
algorithm for building suprema and infima. This allows us to easily obtain
the following:

\begin{theorem}\label{theor:distr}
The lattice $\mathbf{Peq}$ is distributive.
\end{theorem}

\begin{proof}
Let $X$, $Y$, and $Z$ be coinfinite primitive recursive sets. As discussed
above, by $R_X \vee R_Y$ we denote the supremum of $R_{X}$ and $R_Y$, and
$R_X \wedge R_Y$ is the infimum of $R_X$ and $R_Y$. We sketch the proof for
the following distributivity law:
\[
	R_X \vee (R_Y \wedge R_Z) \equiv_{pr} (R_X \vee R_Y) \wedge (R_X \vee R_Z).
\]

Suppose that $Q \equiv_{pr} R_X \vee (R_Y \wedge R_Z)$ and $S \equiv_{pr}
(R_X \vee R_Y) \wedge (R_X \vee R_Z)$. We may assume that $Q = R_U$ and $S =
R_V$, where $U$ and $V$ are primitive recursive sets such that $0\in U\cap
V$, and for every $s\in\omega$,
\begin{gather*}
\#(0^U)[s]  = \max ( \#(0^X)[s],
\min (\#(0^Y)[s], \#(0^Z)[s] ) ),\\
\#(0^V)[s] = \min ( \max(\#(0^X)[s], \#(0^Y)[s]),
\max( \#(0^X)[s], \#(0^Z)[s] ) ).
\end{gather*}

Since the structure $(\omega, \leq)$ is a linear order, for any numbers
$x,y,z\in\omega$, we have
\[
	\max(x, \min(y,z)) = \min( \max(x,y), \max (x,z) ).
\]
Hence, it is clear that $\#(0^U)[s] = \#(0^V)[s]$ for every $s$, and $R_U =
R_V$. This concludes the proof of Theorem~\ref{theor:distr}.
\end{proof}

\section{Density}
We prove now that the distributive lattice of punctual degrees is  dense.
This contrasts with the case of \textbf{Ceers} and \textbf{ER}, where each
degree has a  minimal cover (see \cite{Andrews-Sorbi-19,andrews2021structure}
for details). However, density is a phenomenon that often shows up when
focusing on the subrecursive world. Mehlorn \cite{mehlhorn1976polynomial}
proved that  the degree structures induced by many subrecursive
reducibilities on sets (including the primitive recursive one) are dense.
Similarly, Ladner~\cite{ladner1975structure} proved that, if $\mathbf{P}\neq
\mathbf{NP}$, then the poset of $\mathbf{NP}$ sets under polynomial-time
reducibility is dense.

Density emerges also in  the study of the online content of structures. More
precisely, for a structure $\mathcal{A}$, $\mathbf{FPR}(\mathcal{A})$ denotes
the degree structure generated by primitive recursive isomorphisms on the
collection of all punctual copies of $\mathcal{A}$.  Bazhenov, Kalimullin,
Melnikov, and Ng~\cite{bazhenov2020online} recently proved the following: if
a punctual infinite $\mathcal{A}$ is finitely generated, then the poset
$\mathbf{FPR}(\mathcal{A})$ is either one-element or dense.

\begin{theorem}[Density]\label{thm:density}
If $R_{X}<_{pr} R_{Z}$ are punctual equivalence relations then there exists a
primitive recursive set $Y$ such that $R_{X}<_{pr} R_{Y} <_{pr}R_{Z}$.
\end{theorem}

\begin{proof}
We will satisfy the following requirements, for every $e\in \omega$:
 \begin{align*}
 &P_e:  \text{$p_e$ does not reduce $R_{Y}$ to $R_{X}$,}\\
  &Q_{e}: \text{$p_e$ does not reduce $R_{Z}$ to $R_{Y}$,}\\
  &M: R_X \leq_{pr} R_{Y},\\
  &N: R_{Y} \leq_{pr} R_{Z},
\end{align*}
where $\set{p_e}_{e\in\omega}$ is a computable listing of all primitive
recursive functions, see Remark~\ref{ssct:listing}.

Assume that $f: R_{X} \leq_{pr} R_{Z}$.  Assume also, without losing
generality, that $0 \in X\cap Z$, and that $Z$ is both infinite and co-infinite.

\bigskip
\subsection*{The environment}
At stage $s+1$ we inherit from stage $s$ a finite binary string
$\sigma^{Y}_{s}$ of length $s+1$; moreover we will let
$\sigma^{X}_{s}=X\restriction s+1$ and $\sigma^{Z}_{s}=Z\restriction s+1$.
For $U \in \{X, Y, Z\}$ we will denote $\ZeU[s]=\ZeUS$ (see
Section~\ref{ssct:strings} for the notation $\#\left(0^{\tau}\right)$, where
$\tau$ is any finite binary string).

\bigskip
\subsection*{The strategies}

Let us sketch the strategy to achieve $R_{Y} \nleq_{pr} R_{X}$. When we
attack for the first time the requirement at stage $s_{0}$ we are given the
strings $\sigma^{X}_{s_{0}}, \sigma^{Y}_{s_{0}}, \sigma^{Z}_{s_{0}}$, for
which we have guaranteed that $\ZeY[s_{0}]=\ZeZ[s_{0}]$.

We open the so called $P_{e}$-cycle: until $p_{e}$ does not show a
counterexample to $R_{Y} \leq_{pr} R_{X}$, we keep copying larger and larger
pieces of $Z$ in $Y$, so that starting from the input $s_0+1$  the set $Y$
looks like $Z$ from the input $s_0+1$. If this process goes on forever, then
we would eventually get $R_{Z} \leq_{pr} R_{Y}$: the initial segment
$\sigma^{Z}_{s_{0}}$ of $Z$ which is not copied by the copying procedure can
be mapped by the reduction to $\sigma^{Y}_{s_{0}}$ as the two strings have
the same number of $0$'s; if $i<s_0+1$ is such that $\sigma^{Z}_{s_{0}}(i)=1$
(i.e.\ $Z(i)=1$)) then the reduction maps $i$ to $0 \in Y$. Thus, eventually
we get that $p_{e}$ does show a counterexample to $R_{Y} \leq_{pr} R_{X}$,
otherwise $R_{Y} \leq_{pr} R_{X}$, but then $R_{Z} \leq_{pr} R_{Y} \leq_{pr}
R_{X}$.

When a counterexample shows up, we close the $P_{e}$-cycle and we move to
next requirement, opening the $Q_{e}$-cycle. (In fact before opening the
$Q_{e}$-cycle, we have to go through a transition phase to reach a stage $t$
at which $\ZeY[t]=\ZeX[t]$.) This also shows that $P_{e}$ is eventually
satisfied.

The strategy to achieve
$R_{Z} \nleq_{pr} R_{Y}$ is similar, opening and closing the so called
$Q_{e}$-cycle:  until $p_{e}$ does not show a counterexample to $R_{Z}
\leq_{pr} R_{Y}$ we keep copying larger and larger pieces of $X$ in $Y$, so
that starting from the input $s_0+1$ (where $s_0$ is when the cycle was
opened) the set $Y$ looks like $X$ from $s_0+1$.
In order to implement this procedure in a correct way, we require the following: when
we start the $Q_{e}$-cycle at $s_{0}$, we have $\ZeY[s_{0}]=\ZeX[s_{0}]$.

Again, the $Q_{e}$-cycle cannot go on forever, otherwise we would get
$R_{Z}\leq_{pr} R_{Y}$ (since $p_{e}$ never shows a counterexample), but on
the other hand the copying procedure would give $R_{Y} \leq_{pr} R_{X}$,
yielding a contradiction. After a counterexample shows up, there will be a
transition phase, at the end of which we will reach a stage $t$ at which
$\ZeY[t]=\ZeZ[t]$.

It remains to explain how we achieve that $R_{X}\leq_{pr} R_{Y} \leq_{pr}
R_{Z}$. For this, we guarantee that at each step $s$ we have
\[
\ZeX[s]\leq \ZeY[s] \leq \ZeZ[s],
\]
so that we can search in a bounded way for the images in $Y$ of the $0$'s in
$\sigma^{X}_{s}$, and for the images in $Z$ of the $0$'s in $\sigma^{Y}_{s}$.
This, together with the facts that $s$ is in the domains of both
$\sigma^{X}_{s}$ and $\sigma^{Y}_{s}$, and the mappings $s\mapsto
\sigma^{U}_{s}$ are primitive recursive, will give the desired reductions.

\begin{remark}\label{rem:ass1}
As $R_X \leq_{pr} R_Z$, we may assume that if $\sigma \subset X, \tau \subset
Z$ have the same length then $\ZeS\leq \ZeA$: for this, one can replace $Z$ with the join $X
\lor Z$ if needed.
\end{remark}

\subsection*{The construction}
The construction is in stages.

\medskip

\subsubsection*{Stage $0$}
Let $Y(0)=1$, and $\sigma^{Y}_{0}=\lambda$. \emph{Open the $P_{0}$-cycle}.
Notice that $\ZeX[0]= \ZeY[0]=\ZeZ[0]=0$.

\bigskip

\subsubsection*{Step $s+1$} We distinguish the two relevant cases:

\smallskip

\noindent\textsc{Case 1)} Suppose that we are within a previously opened $P_{e}$-cycle which
has not been declared closed yet. We assume by induction that when we
opened (say at $s_{0}$) the cycle, we had $\ZeX[s_{0}]\leq
\ZeY[s_{0}]=\ZeZ[s_{0}]$.

\bigskip
\subsection*{Copying phase}
(Copy $R_Z$ in $R_Y$.) If we have not yet moved to the $P_e\rightarrow
Q_e$-transition phase, then let
\[
\sigma^Y_{s+1}=\catt{\sigma^Y_s}{Z(s+1)}.
\]
Notice that by the assumption in Remark~\ref{rem:ass1} after this we still
have
\[
\ZeX[s+1] \leq \ZeY[s+1]=\ZeZ[s+1].
\]
After this, if $p_{e}$ has shown a counterexample to $R_{Y}\leq_{pr} R_{X}$
(as defined in Section~\ref{ssct:counterexample}) then enter the
\emph{$P_e\to Q_e$-transition phase}:

\subsection*{Transition phase}
Carry out the following.
\begin{enumerate}
\item If $\ZeX[s+1] = \ZeY[s+1]$ then exit from the transition phase. We
    \emph{close the $P_{e}$-cycle} and \emph{open the $Q_{e}$-cycle}.

\item If $\ZeX[s+1] < \ZeY[s+1]$ then let
\[
\sigma^Y_{s+1}=\catt{\sigma^Y_s}{1}:
\]
(when $X(s+1)=0$ this has the effect of making $\ZeX[s+1]=\ZeX[s]+1$,
whereas $\ZeY[s+1]=\ZeY[s]$) and go to (1), remaining in this transition
phase.
\end{enumerate}
Notice  that at each stage $t$ within  a $P_{e}$-cycle we have by the
assumption in Remark~\ref{rem:ass1}
\[
\ZeX[t] \leq \ZeY[t] \leq \ZeZ[t],
\]
and when we close the $P_{e}$-cycle, we have
\[
\ZeX[t] = \ZeY[t] \leq \ZeZ[t].
\]

\bigskip

\noindent\textsc{Case 2)} Suppose that we are within a previously opened $Q_{e}$-cycle which
has not been declared closed yet. We assume by induction that when we
opened (say at $s_{0}$) the cycle we had $\ZeX[s_{0}]=\ZeY[s_{0}]\leq
\ZeZ[s_{0}]$.

\subsection*{Copying phase}
(Copy $R_X$ in $R_Y$.) Let
\[
\sigma^Y_{s+1}=\catt{\sigma^Y_s}{X(s+1)}.
\]
Notice that by the assumption in Remark~\ref{rem:ass1} after this we still
have $\ZeX[s+1]= \ZeY[s+1]\leq \ZeZ[s+1]$ if we had $\ZeX[s]= \ZeY[s]\leq
\ZeZ[s]$. After this, if $p_{e}$ has shown a counterexample to
$R_{Z}\leq_{pr} R_{Y}$ (as defined in Section~\ref{ssct:counterexample}) then
enter the \emph{$Q_e\to P_{e+1}$-transition phase}:

\subsection*{Transition phase}
Carry out the following.
\begin{enumerate}
\item If $\ZeY[s+1]=\ZeZ[s+1]$ then exit from the transition phase. We
    \emph{close the $Q_{e}$-cycle} and \emph{open the $P_{e+1}$-cycle}.

\item If $\ZeY[s+1]<\ZeZ[s+1]$ then let
\[
\sigma^Y_{s+1}=\catt{\sigma^Y_s}{0}:
\]
(when $Z(s+1)=1$ this has the effect of making $\ZeY[s+1]=\ZeY[s]+1$
whereas $\ZeZ[s+1]=\ZeZ[s]$) and go to (1), remaining in this transition
phase.
\end{enumerate}
Notice that at each stage $t$ within a $Q_{e}$-cycle we have by the
assumption in Remark~\ref{rem:ass1}
\[
\ZeX[t] \leq \ZeY[t]\leq \ZeZ[t],
\]
and when we close the $Q_e$-cycle we have
\[
\ZeX[t] \leq \ZeY[t]= \ZeZ[t].
\]

\subsection*{The verification} The verification relies on the following lemmas.

\begin{lemma}\label{lem:density-reqs}
For each $e$, the requirements $P_{e}$ and $Q_{e}$ are satisfied.
\end{lemma}

\begin{proof}
As in the proof of Theorem~\ref{thm:incomparability}, it easily follows by
induction that the $P_{e}$-cycle is opened at stage $0$ if $e=0$, and at the
stage at which the $Q_{e-1}$-cycle is closed if $e>0$. The $Q_{e}$-cycle
is opened at the stage at which the $P_{e}$-cycle is closed.

Assume that for
every $i<e$ the $P_{i}$-cycle and the $Q_{i}$-cycle have been closed, and the
corresponding requirements are satisfied. Then at the stage $s_{0}$ (with
$s_{0}=0$ if $e=0$, or $s_{0}$ is the stage when we close the $Q_{e-1}$-cycle
if $e>0$) we open the $P_{e}$-cycle. If $p_{e}$ never shows a counterexample
to $R_{Y} \leq_{pr} R_{X}$ then we claim that $R_{Z} \leq_{pr} R_{X}$, a
contradiction.

To show this claim, notice that in this case (i.e.\ should
$p_{e}$ never show a counterexample to $R_{Y} \leq_{pr} R_{X}$, implying that
$R_{Y}\leq_{pr} R_{X}$) we would have $Y= \sigma^Y_{s_0} \ast Z$. Then $R_{Z}
\leq_{pr} R_{Y}$ by a  primitive recursive function $q$ which matches up the
zeros in $\sigma^{Z}_{s_{0}}$ with those of $\sigma^{Y}_{s_{0}}$ (using that
both strings have the same number of zeros, since $\ZeY[s_{0}]=\ZeZ[s_{0}]$),
$q(i)=0$ if $Z(i)=1$ and $i\leq s_0$, and $q(i)=i$ for $i \ge s_0+1$. It
would follow that $R_{Z} \leq_{pr} R_{X}$, a contradiction.

Thus, at some
stage $p_{e}$ shows a counterexample to $R_{Y} \leq_{pr} R_{X}$, whence $P_{e}$ is
satisfied. Moreover, since $\overline{X}$ is infinite the transition phase of
the cycle will end, since eventually $X$ will produce enough $0$'s to match
up with those which are present in $\sigma^{Y}$ at the beginning of the
$P_{e}\to Q_{e}$-transition phase of the $P_{e}$-cycle. Therefore, the cycle
will be closed.

Similarly, assume that for every $i<e$ the $Q_{i}$-cycle has been closed, and
for every $i\leq e$ the $P_{i}$-cycle has been closed, and the corresponding
requirements are satisfied. If $s_{0}$ is the stage at which we open the
$Q_{e}$-cycle (that is when we close the $P_{e}$-cycle) and for every $s\ge
s_{0}$ we never close the cycle, then $R_{Z} \leq_{pr} R_{Y}$, and thus an
argument similar to the one given above would entail that $R_{Z}$ would be
$pr$-reducible to $R_{X}$, as the construction would ensure in this case that
$Y=\sigma^Y_{s_0}\ast X$ and $\ZeX[s_{0}]=\ZeY[s_{0}]$, giving that $R_Y
\leq_{pr} R_X$. Finally, the $Q_e\to P_{e+1}$-transition phase ends, since $Z$
is infinite.

Hence, all $P$- and $Q$- requirements are satisfied.
\end{proof}

\begin{claim}
$Y$ is primitive recursive.
\end{claim}
\begin{proof}
The function $s\mapsto \sigma^{Y}_s$ is primitive recursive and
$Y(s)=\sigma^{Y}_{s+1}(s)$.
\end{proof}

\begin{lemma}
$R_{X} \leq_{pr} R_{Y} \leq_{pr} R_{Z}$.
\end{lemma}

\begin{proof}
We need to define two primitive recursive functions $g,h$ which provide
reductions $g \colon R_{X}\leq_{pr} R_{Y}$ and $h\colon R_{Y} \leq_{pr} R_{Z}$. Using
that the functions $q^{Y}, q^{Z}$ where $q^{Y}(s)= \sigma^{Y}_{s}$ and
$q^{Z}(s)=\sigma^{Z}_{s}$ are primitive recursive, and at each stage $t$ we
have that $\ZeX[t]\leq \ZeY[t]$ define
\[
g(s)=
\begin{cases}
0, &\textrm{if $X(s)=1$},\\
\min \{i < l^{Y}_{s}: \sigma^{Y}_{s}(i) =0\,\&\, (\forall j<s)[i\ne g(j)]\},
&\textrm{if $X(s)=0$}.
\end{cases}
\]
Similarly, using that at each stage $t$ we have $\ZeY[t]\leq  \ZeZ[t]$ we can
define
\[
h(s)=
\begin{cases}
0, &\textrm{if $Y(s)=1$},\\
\min \{i< l^{Z}_{s}: \sigma^{Z}_{s}(i)=0\,\&\,
(\forall j<s)[i\ne h(j)]\}, &\textrm{if $Y(s)=0$}.
\end{cases}
\]
It is not hard to see that $g$ and $h$ provide the desired $pr$-reductions.
\end{proof}

The last lemma ensures that the global requirements $M$ and $N$ are both satisfied. In combination with Lemma~\ref{lem:density-reqs}, this means that $R_Y$ lies strictly in between $R_X$ and $R_Z$, as desired. Theorem~\ref{thm:density} is proved.
\end{proof}

Upwards and downwards density are immediate consequences of
Theorem~\ref{thm:incomparability}, the existence of infima
(Theorem~\ref{theor:lattice}), and Theorem~\ref{thm:density}:

\begin{corollary}\label{cor:up-low}
If $R<_{pr} \Id$, then there are $S_0,S_1$ such that
\[
S_0<_{pr}R<_{pr}S_1<_{pr} \Id.
\]
\end{corollary}

\begin{proof}
Upwards density (i.e.\ existence of $S_{1}$) is a particular case of
Theorem~\ref{thm:density}. For downward density (i.e.\ existence of $S_{0}$),
recall that if $R$ is a punctual equivalence relation, then by
Theorem~\ref{thm:incomparability}, there exists $S$ such that $R \mid_{pr} S$:
thus $S_{0}:= R\wedge S$ is a punctual equivalence relation such that $S_{0}
<_{pr} R$.
\end{proof}

We now combine the density strategy of the previous theorem with the
incomparability strategy exploited in Theorem~\ref{thm:incomparability}.

\begin{theorem}[Density plus incomparability]\label{thm:density-plus-incomparability}
If $R_{X}<_{pr} R_{T} <_{pr} R_{Z}$ are punctual equivalence relations,
then there exists a primitive recursive set $Y$ such that $R_{X}<_{pr} R_{Y}
<_{pr}R_{Z}$ and $R_{T} \mid_{pr} R_{Y}$.
\end{theorem}

\begin{proof}
Suppose that $R_{X}<_{pr} R_{T} <_{pr} R_{Z}$ are punctual equivalence
relations. To build $Y$, a trivial modification of Theorem~\ref{thm:density}
suffices.

In the previous proof, we close the $P_{e}$-cycle in Case~1 of Step~$s+1$ when we see
that $p_{e}$ has shown a counterexample to $R_{Y}\leq_{pr} R_{X}$. For the
purpose of the present proof, we now ask to \emph{close the $P_{e}$-cycle} in
Case~1 of Step~$s+1$ when we have seen that $p_{e}$ has shown  a
counterexample to $R_{Y}\leq_{pr} R_{T}$, and we have matched up through the
transition phase $\ZeX= \ZeY$:  should $p_{e}$ never show a counterexample to
$R_{Y}\leq_{pr} R_{T}$, then (as in the proof of Theorem~\ref{thm:density})
our copying phase of Case ~1 would end up with making $R_{Z} \leq_{pr} R_{Y}$,
giving $R_{Z}\leq_{pr} R_{T}$, a contradiction.

Similarly, here we ask to \emph{close
the $Q_{e}$-cycle} in Case~2 of Step~$s+1$ when we see that $p_{e}$ shows  a
counterexample to $R_{T}\leq_{pr} R_{Y}$, and we have matched up through the
transition phase $\ZeY= \ZeZ$. Should $p_{e}$ never show a counterexample to
$R_{T}\leq_{pr} R_{Y}$, then (as in the proof of Theorem~\ref{thm:density})
our copying phase of Case ~2 would end up with making $R_{Y} \leq_{pr} R_{X}$,
giving $R_{T}\leq_{pr} R_{X}$, a contradiction.
\end{proof}

\begin{remark}\label{rem:new-proof}
Notice that the two previous theorems provide another proof of
Theorem~\ref{thm:incomparability}: Indeed, given a punctual $R<_{pr} \mathrm{Id}$, it is enough
to pick by Corollary~\ref{cor:up-low}  $S_{0}, S_{1}$ such that $S_{0} <_{pr}
R <_{pr} S_{1}$, so that by density plus incomparability there exists
$S \mid_{pr} R$, with the stronger specification that $S$ lies between $S_{0}$
and $S_{1}$.

Notice also that by a straightforward extension of the argument in
Theorem~\ref{thm:density-plus-incomparability} (in the same vein as in the
argument for Corollary~\ref{cor:antichain}), one can show that if $R<_{pr} S$
then one can build an infinite antichain whose members all lie between $R$
and $S$.
\end{remark}

\section{Join- and meet-reducibility}

In a poset $\langle P, \leq \rangle$ an element $a \in P$ is
\emph{join-reducible} if in $P$ there are $b, c < a$ such that $a$ is the
join of $b,c$, and  $a$ is \emph{meet-reducible} if there are $b, c > a$ such
that $a$ is the meet of $b,c$.

Before showing that in $\Peq$ every element is join-reducible, and every $R<_{pr} \mathrm{Id}$ is
meet-reducible, let us introduce some notations and simple observations which
will be useful in the rest of this section.

In analogy with the principal function $p_{\overline{X}}$ of the complement
$\overline{X}$ (where $X\subseteq \omega$), given a string $\sigma \in
2^{<\omega}$, let also $p_{\overline{\sigma}}$ denote the order preserving
finite bijection $p_{\overline{\sigma}}: \{n: n < \ZeS\} \longrightarrow
\overline{\sigma}$ (the notation $\overline{\sigma}$ has been introduced in
Section~\ref{ssct:strings}). Again in analogy with what we have done for sets
(see Definition~\ref{def:joinsets}), we give the following definition.

\begin{emdef}
Given $\sigma, \tau \in 2^{<\omega}$ such that $l^{\sigma}=l^{\tau}=h$ and
$\ZeS=\ZeA=m$ let $\sigma \lor \tau$ be the string with $l^{\sigma \lor
\tau}=h$ and such that $(\sigma \lor \tau)(i)=0$ if and only if
$i=\min(p_{\overline{\sigma}}(n), p_{\overline{\tau}}(n))$, for some $n<m$.
Dually, define $\sigma \wedge \tau$ to be the string with $l^{\sigma \wedge
\tau}=h$ and such that $(\sigma \wedge \tau)(i)=0$ if and only if
$i=\max(p_{\overline{\sigma}}(n), p_{\overline{\tau}}(n))$, for some $n<m$.
\end{emdef}
Notice that for $\sigma, \tau$ as in the definition, we have
$\ZeSlA=\ZeSwA=m$.

\begin{lemma}\label{lem:simple-obs-on-wedge-lor}
Let $(\sigma_{0}, \sigma_{1})$ be a pair of strings such that $l^{\sigma_0}=
l^{\sigma_1}=h$ and $\ZeSz=\ZeSu$; let $(\tau_{0}, \tau_{1})$ be another pair
of strings such that $l^{\tau_0}= l^{\tau_1}=h'$ and $\ZeAz=\ZeAu$; finally,
let $Y_0,Y_1$ be a pair of sets such that $\sigma_0 \subset Y_0$ and
$\sigma_1 \subset Y_1$. Then, for an operation $\square \in \{\lor, \wedge\}$,
\begin{enumerate}
  \item for every $i<h'$ we have
\[
(\sigma_{0}
\widehat{\phantom{\alpha}} \tau_{0} \square \sigma_{1}
\widehat{\phantom{\alpha}} \tau_{1})(h+i)=(\tau_0 \square \tau_1)(i);
      \]
\item $\sigma_{0} \square \sigma_{1} \subset Y_0 \square Y_1$, and for
    every $i<h$, we have
      \[
(Y_{0} \square Y_{1})(i)=(\sigma_{0} \square \sigma_{1})(i).
      \]
\end{enumerate}
\end{lemma}

\begin{proof}
The proof is immediate. Let $\ZeSz=\ZeSu=m$, and $\ZeAz=\ZeAu=m'$. Item (1)
follows from the fact that $\sigma_{0} \widehat{\phantom{\alpha}} \tau_{0}
\square \sigma_{1} \widehat{\phantom{\alpha}} \tau_{1}$ has $m+m'$ zeros: the
first $m$ ones of them (in order of magnitude) come from comparing the pairs
$(p_{\overline{\sigma_0}}(n), p_{\overline{\sigma_1}}(n))$ with $n<m$; and
the last $m'$ ones of them come from comparing the pairs
$(p_{\overline{\sigma_{0} \widehat{\phantom{\alpha}} \tau_{0}}}(m+n),
p_{\overline{\sigma_{1} \widehat{\phantom{\alpha}} \tau_{1}}}(m+n)$ with
$n<m'$, which amounts to comparing the pairs $(p_{\overline{\tau_0}}(n),
p_{\overline{\tau_1}}(n))$ with $n<m'$.

Finally (2) follows easily from (1).
\end{proof}

\begin{theorem}\label{thm:joinreducible}
Each punctual $R_Z$ is join-reducible.
\end{theorem}

\begin{proof}
Let $R_{Z}$ be a punctual equivalence relation in normal form. As $\Id
\equiv_{pr} R_{E}$, with $E$ denoting the set of even numbers, we  can always
assume that $Z$ is infinite and coinfinite.

We will build $R_{Y_0},R_{Y_1}<_{pr}R_{Z}$ such that $R_{Y_0}\vee R_{Y_1}=
R_Z$. To do so, we will satisfy the following requirements:
 \begin{align*}
  P_e \colon &\ \text{$p_e$ does not reduce $R_Z$ to $R_{Y_1}$},\\
  Q_e \colon &\  \text{$p_e$ does not reduce $R_Z$ to $R_{Y_0}$},\\
  N \colon &\ R_Z = R_{Y_0} \vee R_{Y_1}.
\end{align*}
Notice that the $N$-requirement is actually requesting that $R_Z = R_{Y_0}
\vee R_{Y_1}$, not just $R_Z \equiv_{pr} R_{Y_0} \vee R_{Y_1}$.

We will build $Y_0,Y_1$ in stages by approximating their characteristic
functions, i.e., $Y_i=\bigcup_{s\in\omega}\sigma^{Y_i}_s$ for
$i\in\set{0,1}$. In the construction at each stage $s$ we will use also the
string $\sigma^Z_s$ which, we recall, is the initial segment of $Z$ with
length $s$.

\subsubsection*{The construction}
We adopt the same terminology and notations as those employed in Theorem
\ref{thm:density}. As in the proof of that theorem, at each stage, the
construction can be either in a \emph{copying phase} or in a \emph{transition
phase}.

\subsubsection*{Stage $0$}
$\sigma^{Y_0}_0=\sigma^{Y_1}_{0}=\lambda$. \emph{Open the $P_{0}$-cycle},
which will be implemented starting from next stage.

\subsubsection*{Stage $s+1$}
We distinguish two cases.

\bigskip

\noindent\textsc{Case 1.}
Suppose that we are within a previously opened cycle $P_e$, which has
not been declared closed yet.  We assume by induction that we have
$\ZeYu[s+1]\leq\ZeYz[s+1]\leq \ZeZ[s+1]$.

\bigskip
\subsection*{Copying phase}
If we have not yet moved to the $P_e\rightarrow Q_e$-transition phase,
then we copy $R_{Z}$ into $R_{Y_{0}}$: let
$\sigma^{Y_0}_{s+1}=\catt{\sigma^{Y_0}_s}{Z(s)}$ and
$\sigma^{Y_1}_{s+1}=\catt{\sigma^{Y_1}_s}{1}$. After this, if $p_e$ shows a
counterexample to $R_{Z}\leq_{pr} R_{Y_{1}}$ then go to the
\emph{$P_e\rightarrow Q_e$-transition phase} which will be implemented
starting from the next stage.

\bigskip
\subsection*{Transition phase}
Suppose that we are within the $P_e\rightarrow Q_{e}$-transition phase. Let
$Y_1(s+1)=0$ and $Y_0(s+1)=Z(s)$. After this, if $\ZeYz[s+1]=\ZeYu[s+1]$ then
\emph{close the $P_{e}$-cycle} and \emph{open the $Q_{e}$-cycle} which will
be implemented starting from next stage; otherwise, stay in this transition
phase.

\bigskip

\noindent\textsc{Case 2.}
Suppose that we are within a previously opened $Q_e$-cycle, which has
not been declared closed yet. We assume by induction that we have
$\ZeYz[s+1]\leq\ZeYu[s+1]\leq \ZeZ[s+1]$.

\bigskip
\subsection*{Copying phase}
If we have not yet moved to the $Q_e\rightarrow P_{e+1}$-transition phase,
then we copy $R_{Z}$ into $R_{Y_{1}}$: let $\sigma^{Y_0}_{s+1}=
\catt{\sigma^{Y_0}_s}{1}$ and
$\sigma^{Y_1}_{s+1}=\catt{\sigma^{Y_1}_s}{Z(s)}$. After this, if $p_e$ shows
a counterexample to $R_{Z}\leq_{pr} R_{Y_{0}}$ then go to the
\emph{$Q_e\rightarrow P_{e+1}$-transition phase} which will be implemented
starting from the next stage; otherwise, stay in this transition phase.

\bigskip
\subsection*{Transition phase}
Suppose that we are within the $Q_e\rightarrow P_{e+1}$-transition phase. Let
$Y_0(s+1)=0$ and $Y_1(s+1)=Z(s)$. After this, if $\ZeYz[s+1]=\ZeYu[s+1]$ then
\emph{close the $Q_{e}$-cycle} and \emph{open the $P_{e+1}$-cycle} which will
be implemented starting from next stage; otherwise, stay in this transition
phase.

\bigskip

Notice that for every $s$, the constructed initial segment $\sigma^{Y_i}_s$ of
$Y_i$ has length $s$.

(We note that the distinction between ``copying phase'' and ``transition
phase'' can be misleading, as in the transition phase of Case 1 we still keep
copying $Z$ into $Y_0$ as we were doing during the copying phase, and
similarly in the transition phase of Case 2 we still keep copying $Z$ into
$Y_1$ as we were doing during the copying phase.)

\subsection*{The verification}

$Y_0$ and $Y_1$ are primitive recursive as $Y_i(s)=\sigma^{Y_i}_{s+1}(s)$,
and the mapping $s \mapsto \sigma^{Y_i}_{s+1}$ is primitive recursive.

The rest of the verification is based on the following lemmas.

\begin{lemma}
The $P$- and $Q$- requirements are satisfied. Moreover, if $s$ is a stage at
which we close a cycle, then $\ZeYz[s]=\ZeYu[s]$.
 \end{lemma}

\begin{proof}
As in the proof of Theorem~\ref{thm:incomparability} and
Theorem~\ref{thm:density}, if $s$ is any stage, then at $s$ we are either in
an open $P_{e}$- or $Q_{e}$-cycle for exactly one $e$.

Eventually any $P$- or $Q$- cycle will be closed. This is easily seen by
induction.  Suppose that at stage $s_{0}$ we open the $P_{e}$-cycle (the
$P_{0}$-cycle is opened at stage $0$). Then eventually $p_{e}$ shows a
counterexample to $R_{Z} \leq_{pr} R_{Y_{1}}$ (thus $P_{e}$ is satisfied),
otherwise our copying procedure would put all fresh elements into $Y_1$,
giving that $R_{Y_{1}}$ is finite, contradicting that $R_{Z}$ is punctual and
thus non-finite. After $p_{e}$ has shown a counterexample, we start the
$P_{e}\to Q_{e}$-transition phase. During the transition we keep all fresh
elements out of $Y_1$. This makes the number of $0$'s of $Y_1$ growing as
fast as possible, while we copy $Z$ in $Y_0$. Eventually, we will witness a
stage $s$ at which $\ZeYz[s]=\ZeYu[s]$; otherwise, $Z$ would be finite
contradicting the fact that $Z$ is infinite. This shows that the
$P_{e}$-cycle is eventually closed, and we open the $P_{e}$-cycle.

By a similar argument we can prove that each $Q_{e}$-cycle is eventually
opened, then closed, and the corresponding requirement is satisfied.
\end{proof}

\begin{lemma}
The $N$-requirement is satisfied.
\end{lemma}

\begin{proof}
We want to show that $Z=Y_0 \vee Y_1$. Let $s_{0}< s_{1}< \ldots$ be the
sequence of stages at which we open a cycle, with $s_{0}=0$. Our argument is
by induction on the index $n$ of $s_{n}$. Assume by induction that, for every
$i< s_{n}$ we have $(Y_{0} \lor Y_{1})(i)=Z(i)$: this is true if $n=0$.
Assume that at $s_{n}$ we open a $P_{e}$-cycle, the other case being similar:
this cycle is closed at the stage $s_{n+1}$.

By construction,
$\sigma^{Y_{0}}_{s_{n+1}}=\sigma^{Y_{0}}_{s_{n}} \widehat{\phantom{\alpha}}
\tau_{0}$ and $\sigma^{Y_{1}}_{s_{n+1}}=\sigma^{Y_{1}}_{s_{n}}
\widehat{\phantom{\alpha}} \tau_{1}$, where
\[
	\tau_{0}=\str{Z(s_{n}),\dots, Z(s_{n+1}-1)}
\text{ and } \tau_{1}= 1^h \widehat{\phantom{\alpha}} 0^k,
\]
for some $h,k$ with $h+k=s_{n+1}-s_{n}$ (here, for $i \in \{0,1\}$, $i^m$
denotes the string of length $m$ giving value $i$ on all its inputs). As in
$\tau_{1}$ the bit $0$ appears only in the final segment $0^{k}$, for every
$i< s_{n+1}-s_{n}$ we have that
\[
(\tau_{0} \lor \tau_{1})(i)=Z(s_{n}+i).
\]
It then follows from Lemma~\ref{lem:simple-obs-on-wedge-lor} that for every
$j < s_{n+1}-s_n$ we have
\[
(Y_{0} \lor Y_{1})(s_n+j)=(\tau_{0} \lor \tau_{1})(j)=Z(s_{n}+j),
\]
giving that $(Y_0\lor Y_1)(i)=Z(i)$ for all $i < s_{n+1}$.
\end{proof}
This concludes the verification.
\end{proof}

By a symmetric argument, one can also show the following.

\begin{theorem}\label{thm:meetreducible}
Any $R_{Z} <_{pr} \Id$ is meet-reducible.
\end{theorem}

\begin{proof}
The requirements are
 \begin{align*}
  P_e\colon &\ \text{$p_e$ does not reduce $R_{Y_1}$ to $R_Z$},\\
   Q_e\colon &\ \text{$p_e$ does not reduce $R_{Y_0}$ to $R_Z$},\\
  M \colon &\ R_Z = R_{Y_0} \wedge R_{Y_1}.
\end{align*}
The proof and the construction are similar to the previous theorem, with the
modifications that whenever in the previous theorem in a copying or
transition phase we added the bit $i$ to $Y_0$ or $Y_1$, we now add the bit
$1-i$. Notice for instance that for every $e$, $p_{e}$ eventually shows a
counterexample to $R_{Y_1}\leq_{pr} R_{Z}$ as otherwise now $Y_{1}$ would be
eventually finite, thus $R_{Y_1}\equiv_{pr} \Id$, and thus $\Id \leq_{pr}
R_{Z}$, a contradiction. So $P_e$ is satisfied. A similar argument shows that
each $Q_e$ is satisfied.

In order to show that $R_Z=R_X \wedge R_Y$, notice that this time (assuming that at
$s_n$ we open a $P_e$-cycle, the other case being similar)
$\sigma^{Y_{0}}_{s_{n+1}}=\sigma^{Y_{0}}_{s_{n}} \widehat{\phantom{\alpha}}
\tau_{0}$ and $\sigma^{Y_{1}}_{s_{n+1}}=\sigma^{Y_{1}}_{s_{n}}
\widehat{\phantom{\alpha}} \tau_{1}$ where $\tau_{1}= 0^h
\widehat{\phantom{\alpha}} 1^k$, for some $h,k$ with $h+k=s_{n+1}-s_{n}$, and
$\tau_{0}=\str{Z(s_{n}),\dots, Z(s_{n+1}-1)}$. It then follows from
Lemma~\ref{lem:simple-obs-on-wedge-lor} (as in $\tau_1$ the $0$'s show up
before the $1$'s) that for every $j < s_{n+1}-s_n$ we have
\[
(Y_{0} \wedge Y_{1})(s_n+j)=(\tau_{0} \wedge \tau_{1})(j)=Z(s_{n}+j).
\]
Thus by induction on the index $n$ of $s_n$ we can  show that  $(Y_0\wedge
Y_1)(i)=Z(i)$ for all $i < s_{n}$.
\end{proof}

\section{Embedding of the diamond lattice}
So far, we highlighted that $\mathbf{Peq}$ is a  remarkably well-behaved
degree structure, being a dense distributive lattice. Moreover, we proved
that  degrees below the top are not distinguishable with respect to join- or
meet-reducibility. In these two remaining sections, we will turn the
perspective upside down, focusing on some fairly unexpected ill-behaviour of
$\mathbf{Peq}$. In particular, in this section we show that some intervals of
punctual equivalence relations $R<_{pr} S$ embed the diamond lattice, and
some other don't. In fact, we will offer a complete characterization of the
intervals which embeds the diamond lattice, by relying on a combinatorial
property of primitive recursive sets $X$ and $Y$, named
$\blacklozenge$-property, which intuitively says that there are infinitely
many initial segments of the natural numbers up to which $X$ and $Y$ have an
equivalent number of zeros.

Given a set $U$, throughout the section we agree, as in the proof of
Theorem~\ref{thm:density}, that $\sigma^{U}_{s}$ denotes the initial segment
of $U$ having length $s+1$ and $\ZeU[s]$ denotes the cardinality of
$\overline{\sigma^{U}_{s}}$. To avoid trivial cases, we also assume that here
we consider only primitive sets that are infinite and coinfinite.

\begin{emdef}
We say that a pair $(X,Y)$ of primitive recursive sets satisfies \emph{the
$\blacklozenge$-property} if there is a pair $(X^{*}, Y^{*})$ of primitive
recursive sets such that $R_{X^*}\equiv_{pr} R_{X}$ and $R_{Y^*}\equiv_{pr}
R_{Y}$ and
\[
(\forall s)(\exists t\geq s) \bigg[ \ZeXs[t]=\ZeYs[t] \bigg].
\]
\end{emdef}

We say that a stage $s$ is an \emph{equilibrium point} for a pair $(X,Y)$ of
primitive recursive sets if
\[
\ZeX[s]=\ZeY[s].
\]

\begin{theorem}\label{thm:notdelta}
Suppose that a pair $(X, Z)$ does not satisfy the $\blacklozenge$-property,
and $R_X <_{pr} R_Z$. Then, there are no $Y_0$ and $Y_1$ such that
\[
R_{X} <_{pr} R_{Y_i} <_{pr} R_Z, \ R_{X} \equiv_{pr} R_{Y_0} \wedge R_{Y_1}, \text{ and } R_{Z} \equiv_{pr} R_{Y_0} \vee R_{Y_1}.
\]
\end{theorem}

\begin{proof}
We show the contrapositive statement. Suppose that $R_X$, $R_{Y_0}$,
$R_{Y_1}$, and $R_Z$ form a diamond. Without loss of generality, one may
assume that $0\in Y_0$.

\begin{lemma}
The pair $(Y_0,Y_1)$ has infinitely many equilibrium points.
\end{lemma}
\begin{proof}
Suppose that $s^*$ is the last equilibrium point for $(Y_0,Y_1)$. Then,
without loss of generality, we may assume that for any $s > s^*$,
\[
\ZeYz[s] > \ZeYu[s].
\]

Let $n^*:= \ZeYz[s^*+1]$. For a number $m \ge n^{\ast}$, as
$p_{\overline{Y_0}}(m)> s^{*}+1$, we have
\[
	m+1 = \ZeYz[p_{\overline{Y_0}}(m)] > \ZeYu[p_{\overline{Y_0}}(m)],
\]
and thus $p_{\overline{Y_0}}(m)<p_{\overline{Y_1}}(m)$. By the definition of
$Y_0 \lor Y_1$, we deduce that for every $m \ge n^{\ast}$, we have
$p_{\overline{Y_0 \vee Y_1}}(m) = p_{\overline{Y_0}}(m)$. Therefore, the
function
\[
f(x) :=
\begin{cases}
	0, & \text{if $x \in Y_{0} \lor Y_{1}$},\\
	p_{\overline{Y_0}}(l) , & \text{if } x = p_{\overline{Y_0 \vee Y_1}}(l)
    \text{ for some } l < n^{\ast},\\
	x, & \text{otherwise},
\end{cases}
\]
provides a $pr$-reduction from $R_{Y_0} \vee R_{Y_1}$ into $R_{Y_0}$, which
gives a contradiction.
\end{proof}

Let $s_0 < s_1 <s_2 < \dots$ be the sequence of all equilibrium points for
$(Y_0,Y_1)$. We choose an infinite subsequence of equilibrium points
\[
s^*_0 < s^*_1 < s^*_2 < \dots
\]
such that for every $i\in\omega$, $\ZeYz[s^*_i] < \ZeYz[s^*_{i+1}]$.

Let $t_i := \ZeYz[s^*_i]$. By Lemma~\ref{lem:simple-obs-on-wedge-lor} it is
clear that
\[
\ZYv[s^*_i] = \ZYw[s^*_i] = t_i,
\]
and hence, the pair $(X,Z)$ satisfies the $\blacklozenge$-property.
Theorem~\ref{thm:notdelta} is proved.
\end{proof}

\smallskip

On the other hand, the next result shows that the $\blacklozenge$-property is
sufficient for embedding the diamond:

\begin{theorem}\label{thm:delta}
Let $R_X<_{pr} R_Z$  such that $(X,Z) $ satisfies the
$\blacklozenge$-property. There are punctual $R_{Y_0},R_{Y_1}$ such that the
infimum (resp.\ supremum) of $R_{Y_0}$ and $R_{Y_1}$ is $pr$-equivalent to
$R_X$ (resp.\ $R_Z$).
\end{theorem}

\begin{proof}
Without loss of generality, assume that $X$ and $Z$ are both infinite,
co-infinite, and $0\in X\cap Z$. Observe also that we may assume that the
following hold
\begin{enumerate}
\item the pair $(X,Z)$ has infinitely many equilibrium points;
\item for all $s$, $\ZeZ[s]\geq \ZeX[s]$.
\end{enumerate}
To see this that this can be assumed, first choose $X,Z$ with infinitely many
equilibrium points; they must exist since $(X,Z)$ has the
$\blacklozenge$-property. Second, replace $Z$ with $X\vee Z$ if needed; note
that if $t$ is an equilibrium point for $(X,Z)$, then by
Lemma~\ref{lem:simple-obs-on-wedge-lor} it should be an equilibrium point for
$(X,X\vee Z)$ as well.

We will build sets $Y_0,Y_1$ in stages, satisfying the following
requirements:
 \begin{align*}
  Q_e \colon &\ \text{$p_e$ does not reduce $R_Z$ to $R_{Y_0}$},\\
  P_e \colon &\ \text{$p_e$ does not reduce $R_Z$ to $R_{Y_1}$},\\
  M  \colon &\  R_{X} = R_ {Y_0} \wedge R_{Y_1},\\
  N \colon &\  R_Z = R_{Y_0} \vee R_{Y_1}.
\end{align*}

It is easy to see that the above requirements are sufficient: in particular,
we do not need to prove that $R_{Y_i}\nleq_{pr} R_X$. Notice also that we
require $R_X$ and $R_Z$ to be in fact equal, and not just $pr$-equivalent, to
$R_ {Y_0} \wedge R_{Y_1}$ and $R_ {Y_0} \vee R_{Y_1}$, respectively:
therefore we get for free that $R_ {Y_0}$ and $R_ {Y_1}$ lie in the interval
determined by $R_X$ and $R_Z$. As always, we will build $Y_0,Y_1$ in stages
by approximating their characteristic functions, i.e.,
$Y_i=\bigcup_{s\in\omega}\sigma^{Y_i}_s$ for $i\in\set{0,1}$. Each string
$\sigma^{Y_i}_s, \sigma^{X}_s, \sigma^{Z}_s$ we define or deal with at a
stage $s$ has length $s+1$.

\subsubsection*{The strategies}
In order to achieve that $p_e$ does not reduce $R_Z$ to $R_{Y_0}$ we open the
$Q_e$-cycle and employ a copying procedure, copying $R_X$ into $R_{Y_0}$,
until we see that $p_e$ shows a counterexample to $R_Z \leq_{pr} R_{Y_0}$.
Meanwhile we employ a corresponding copying procedure, copying $R_Z$ into
$R_{Y_1}$.

When seeing that $p_e$ shows a counterexample at stage, say,
$s+1$, by our assumption on always being $\ZeX[t]\leq \ZeZ[t]$ it may happen
that $\ZeX[s+1] < \ZeZ[s+1]$. If so, before closing the $Q_e$-cycle we open
the so-called $Q_e\to P_e$-transition phase, which consists (still copying
$R_X$ into $R_{Y_0}$ and $R_Z$ into $R_{Y_1}$) in prolonging bit by bit
$\sigma^{Y_0}_{s+1}$ (which the construction has guaranteed to have the same
number of $0$'s as $\sigma^{X}_{s+1}$) with the bits of $X$, and in
prolonging bit by bit $\sigma^{Y_1}_{s+1}$ (which the construction has
guaranteed to have the same number of $0$'s as $\sigma^{Z}_{s+1}$) with the
bits of $Z$, until we reach the next equilibrium point of $(X,Z)$: at this
point we close the $Q_e$-cycle and we open the $P_e$-cycle.

The described procedure
has the goal of making it possible to apply
Lemma~\ref{lem:simple-obs-on-wedge-lor} and conclude that the bits added to
$\sigma^{Y_0}_{s+1}$ and $\sigma^{Y_1}_{s+1}$ since when we opened the
$Q_e$-cycle satisfy
\[
	(\sigma^{Y_0}_{s+1}\lor \sigma^{Y_1}_{s+1})(i)=(X\lor
Z)(i) \text{ and } (\sigma^{Y_0}_{s+1}\wedge \sigma^{Y_1}_{s+1})(i)=(X\wedge Z)(i),
\]
so as to eventually get $Y_0 \lor Y_1= X \lor Z$ and $Y_0 \wedge Y_1= X
\wedge Z$.

In order to achieve that $p_e$ does not reduce $R_Z$ to $R_{Y_1}$, we use (in
an obvious way) a similar strategy, this time copying $R_X$ into $R_{Y_1}$
and $R_Z$ into $R_{Y_0}$; we go into the $P_e \to Q_{e+1}$-transition phase
when $p_e$ shows a counterexample. Finally, after reaching the next
equilibrium point, we close the $P_e$-cycle and open the $Q_{e+1}$-cycle.

\subsection*{The construction} Unless otherwise specified, we adopt the same
terminology and notation employed in Theorem \ref{thm:density}.

\subsubsection*{Stage $0$}
$\sigma^{Y_0}_0=\sigma^{Y_1}_{0}=\str{1}$. \emph{Open the $Q_0$-cycle}.

\subsubsection*{Stage $s+1$}
There are two cases.

\bigskip

\noindent\textsc{Case 1.}
Suppose that we are within a previously opened $Q_e$-cycle, which has
not been declared closed.  We assume by induction that we have
\[
\ZeX[s+1]= \ZeYz[s+1]\leq\ZeYu[s+1]= \ZeZ[s+1],
\]
and the first stage $s'$ of this particular $Q_e$-cycle has the following property
\[
\ZeX[s']= \ZeYz[s'] = \ZeYu[s']= \ZeZ[s'].
\]

\bigskip
\subsection*{Copying phase}
If we have not yet moved to the $Q_e\rightarrow P_e$-transition phase,
then we copy $R_{X}$ into $R_{Y_{0}}$ and copy $R_{Z}$ into $R_{Y_{1}}$ : let
$\sigma^{Y_0}_{s+1}=\catt{\sigma^{Y_0}_s}{X(s+1)}$ and
$\sigma^{Y_1}_{s+1}=\catt{\sigma^{Y_1}_s}{Z(s+1)}$. After this, if $p_e$ shows a
counterexample to $R_{Z}\leq_{pr} R_{Y_{0}}$ then go to the
\emph{$Q_e\rightarrow P_e$-transition phase}, which will be implemented
starting from the next stage.

\bigskip

\subsection*{Transition phase}
Suppose that we are within the
    transition phase of a previously opened $Q_e$-cycle, which has not
been declared closed yet. Let
    $\sigma^{Y_0}_{s+1}=\catt{\sigma^{Y_0}_s}{X(s+1)}$ and
    $\sigma^{Y_1}_{s+1}=\catt{\sigma^{Y_1}_s}{Z(s+1)}$. After this, if
    $\ZeYz[s+1]= \ZeYu[s+1]$, then \emph{close the $Q_e$-cycle}, and
    \emph{open the $P_{e}$-cycle}, which will be processed starting from
    next stage.

\bigskip

\noindent\textsc{Case 2.} Suppose that we are within a previously opened $P_e$-cycle, which has
not been declared closed yet. We assume by induction that
\[
	\ZeX[s+1]= \ZeYu[s+1]\leq\ZeYz[s+1]= \ZeZ[s+1],
\]
and when we had opened this cycle, we had $\ZeYu[s'] = \ZeYz[s']$.

\subsection*{Copying phase} If we have not yet moved to the
    $P_e\to Q_{e+1}$-transition phase, then let $\sigma^{Y_0}_{s+1}=
    \catt{\sigma^{Y_0}_s}{Z(s+1)}$ and
    $\sigma^{Y_1}_{s+1}=\catt{\sigma^{Y_1}_s}{X(s+1)}$. After this, if
    $p_e$ has shown a counterexample for $p_e\colon R_Z \leq_{pr} R_{Y_1}$, then \emph{move to the
    $P_e\to Q_{e+1}$-transition phase}.

\subsection*{Transition phase} Suppose that we are within the
    transition phase of a previously opened $P_e$-cycle, which has not
    been declared closed yet. Let
    $\sigma^{Y_0}_{s+1}=\catt{\sigma^{Y_0}_s}{Z(s+1)}$ and
    $\sigma^{Y_1}_{s+1}=\catt{\sigma^{Y_1}_s}{X(s+1)}$. After this, if
    $\ZeYz[s+1]= \ZeYu[s+1]$ then \emph{close the $P_e$-cycle}, and
    \emph{open the $Q_{e+1}$-cycle}, which will be processed starting from
    next stage.

\subsection*{The verification}

The sets $Y_0, Y_1$ are primitive recursive as $Y_i(s)=\sigma^{Y_i}_{s}(s)$
(recall that $l^{Y_i}_s=s+1$) and the mapping $s \mapsto \sigma^{Y_i}_{s}$ is
primitive recursive. The rest of the verification is based on the following
lemmas.

\begin{lemma}
The $P$- and $Q$- requirements are satisfied.
\end{lemma}

\begin{proof}
Similarly to the proofs of Theorem~\ref{theor:lattice} and Theorem \ref{thm:density},
it is easily seen by induction that every $P$- or $Q$-cycle is opened and
later closed, and exactly one cycle is open at each stage. Consider for
instance a $Q_e$-cycle, and assume that is it was opened at stage $s_0$, and we started
processing the cycle from $s_0+1$. Assume also that
\[
\ZeX[s_0]= \ZeYz[s_0].
\]
Should $p_e$ never show a counterexample to $R_Z \leq_{pr} R_{Y_0}$ then it
would be $R_Z \leq_{pr} R_X$. Indeed, in this case we would eventually get
$Y_0= \sigma^{Y_0}_{s_0} \ast X$ (see Section~\ref{ssct:strings} for the
notation): thus, $p_e\colon R_Z \leq_{pr} R_{Y_0}$ would imply $R_Z \leq_{pr} R_X$.
Therefore, eventually we do get a counterexample, and requirement $Q_e$ is satisfied.

After $p_e$ shows a counterexample, we start the transition phase: when we
open it (say, at $s_1$) we have
\[
\ZeYz[s_1]\leq\ZeYu[s_1].
\]
By our assumption that $(X,Z)$ has the $\blacklozenge$-property it follows
(by prolonging $\sigma^{Y_0}$ as $X$, and $\sigma^{Y_1}$ as $Z$) that
eventually $\ZeYz$ catches up with $\ZeYu$, thus we reach a stage $s+1$ when
$\ZeYz[s+1]=\ZeYu[s+1]$. At this stage, we close the $Q_e$-cycle and we open
the $P_e$-cycle.

A similar claim holds for $P_e$-cycles. Note that in the second part of the cycle,
the transition phase waits until the inequality $\ZeYu[t] \leq \ZeYz[t]$ reaches a stage $s+1$
such that $\ZeYu[s+1]=\ZeYz[s+1]$.

We also conclude that all $P$- and $Q$-requirements are satisfied.
\end{proof}

\begin{lemma}\label{lemma:Mreq}
The $M$-requirement and the $N$-requirement are satisfied.
\end{lemma}

\begin{proof}
Let $0=s_0< s_1< \ldots$ be an infinite sequence of stages $s$ at which we
have
\[
\ZeX[s]= \ZeYz[s]=\ZeYu[s]= \ZeZ[s].
\]
For instance, this happens when we open cycles.

Let $i\in \omega$ and let $n$
be such that $i < s_{n+1}-s_n$.
Suppose that at $s_n$ we open a $Q$-cycle~--- then $\sigma^{Y_0}_{s_{n+1}}= \sigma^{Y_0}_{s_{n}} \widehat{\phantom{\alpha}}
\tau_0$ and $\sigma^{Y_1}_{s_{n+1}}= \sigma^{Y_1}_{s_{n}}
\widehat{\phantom{\alpha}} \tau_1$, where $\tau_0(i)=X(s_n+1+i)$ and
$\tau_1(i)=Z(s_n+1+i)$. Since the pairs $(\sigma^{Y_0}_{s_{n}},
\sigma^{Y_1}_{s_{n}})$, $(\tau_0, \tau_1)$, and $(Y_0,Y_1)$ satisfy the
assumptions of Lemma~\ref{lem:simple-obs-on-wedge-lor}, it follows by
induction on the index $n$ of $s_n$ that
\begin{align*}
(Y_0 \wedge Y_1)(i)&= (X\wedge Z)(i)\\
(Y_0 \lor Y_1)(i)&= (X\lor Z)(i).
\end{align*}
Hence, we deduce that $R_{Y_0} \wedge R_{Y_1} = R_X$ and $R_{Y_0} \lor R_{Y_1} = R_Z$.
Lemma~\ref{lemma:Mreq} is proved.
\end{proof}

This concludes the verification. Theorem~\ref{thm:delta} is proved.
\end{proof}

Combining Theorem \ref{thm:notdelta} and Theorem \ref{thm:delta}, we obtain
the following.

\begin{corollary}\label{cor:deltapropertydefinable}
An interval $[R,S]$ of $\mathbf{Peq}$ embeds the diamond lattice preserving
$0$ and $1$ if and only if $(R,S)$ has the $\blacklozenge$-property.
\end{corollary}

\section{On the intricacy of $\Peq$}
In this final section, we deepen the analysis of $\mathbf{Peq}$, unveiling
further structural complexity. Most notably, we will focus on the
automorphisms of $\Peq$, proving that such a degree structure is neither
rigid nor homogeneous. We will also show that $\Peq$  contains nonisomorphic
lowercones.  These results will require  both to further explore the
consequences of the $\blacklozenge$-property defined above and to introduce
another property, named slowness, concerning the rate at which a primitive
recursive set shows its zeros. We conclude the section by collecting a number
of interesting open questions, which may motivate future work. We are
particularly interested in whether the theory of $\mathbf{Peq}$ is decidable
or not.

\begin{remark} (Redefining the symbol $\ZOP{X}[s]$.)
In this section, for technical reasons, we will take $\ZOP{X}[s]$ to be the
number of $i\leq k$ such that $X(i)=0$, where $k$ is the largest such that
$X(j)\downarrow$ in at most $s$ many steps for all $j\leq k$. So $\ZOP{X}[s]$
is the number of zeroes that we can see in the characteristic function of $X$
after evaluating it for $s$ many steps.
\end{remark}

The next lemma is an analogue of Proposition~\ref{prop:growth-rate}~--- it
expresses $R_X\leq_{pr}R_Y$ in terms of the growth rates of $\ZOP{X}$ and
$\ZOP{Y}$:

\begin{lemma}\label{lem:reductionintermsofcharfunction}
Given any $X,Y$, $R_X\leq_{pr} R_Y$ if and only if there exists a primitive
recursive function $p$ such that for every $s$,
$\ZOP{X}[s]\leq\ZOP{Y}[p(s)]$.
\end{lemma}
\begin{proof}
Suppose that $R_X\leq_{pr}R_Y$ via $f$. For each $s$ we let $p(s)$ be the
least stage $t>s$ such that $Y(n)[t]\downarrow$ for all $n\leq f(s)$. Then
$\ZOP{X}[s]\leq \ZOP{Y}[p(s)]$.

Now conversely fix $p$. For each $m$ we find the first stage $t$ for which we
have $X\upharpoonright (m+1)[t]\downarrow$. Let $h(m)=p(t)$. Then
$h(p_{\overline{X}}(n))\geq p_{\overline{Y}}(n)$ for all $n$ and by
Proposition \ref{prop:growth-rate}, $R_X\leq_{pr} R_Y$.
\end{proof}

It is easy to see that there are $R<_{pr} S$ such that the pair $(R,S)$
satisfies the $\blacklozenge$-property: By Theorem \ref{thm:incomparability}
take a pair of incomparable $Y_0\mid_{pr} Y_1$, then $R=Y_0\wedge Y_1 <_{pr}
S=Y_0\vee Y_1$ has the $\blacklozenge$-property. In fact, by Theorems
\ref{thm:joinreducible} and \ref{thm:meetreducible}, given any $R<_{pr} \Id$
there is some $S>_{pr} R$, and given any $S$ there is some $R<_{pr}S$ such
that $(R,S)$ has the $\blacklozenge$-property. So every punctual degree is
the top and (if it is not $\Id$) the bottom  of an interval with the
$\blacklozenge$-property.

However, since the $\blacklozenge$-property is a property of a $pr$-degree
and not of a set, it is not totally obvious why there should be an interval
that does \emph{not} satisfy the $\blacklozenge$-property. We prove a lemma
which expresses the $\blacklozenge$-property as a property about sets. Note
that the $\blacklozenge$-property does not apriori require the two sets to be
comparable, and the characterization below holds in general.

\begin{lemma}\label{lem:deltapropertysets}
A pair $(X,Y)$ satisfies the $\blacklozenge$-property if and only if there
exist primitive recursive functions $p$ and $q$ such that
$\ZOP{X}[s]=\ZOP{Y}[p(s)]=\ZOP{Y}[t]=\ZOP{X}[q(t)]$ for infinitely many
$s,t$.
\end{lemma}
\begin{proof}
Suppose that $(X,Y)$ satisfies the $\blacklozenge$-property. Fix $(X^{*},
Y^{*})$ witnesseing that the pair $(X,Y)$ has the $\blacklozenge$-property,
so that
\[
(\forall s)(\exists t\geq s)\left[\ZeXs[t]=\ZeYs[t]\right],
\] and functions $f_{X,X^*},f_{X^*,X},f_{Y,Y^*}$ and $f_{Y^*,Y}$ satisfying
\begin{gather*}
	\ZOP{X}[s]\leq\ZOP{X^*}[f_{X,X^*}(s)],\ \ \ZOP{X^*}[s]\leq\ZOP{X}[f_{X^*,X}(s)],\\
	\ZOP{Y}[s]\leq\ZOP{Y^*}[f_{Y,Y^*}(s)], \ \ \ZOP{Y^*}[s]\leq\ZOP{Y}[f_{Y^*,Y}(s)]
\end{gather*}
for every $s$, respectively (applying Lemma
\ref{lem:reductionintermsofcharfunction}). Then obviously we should take $p$
and $q$ to be the composition of the given functions in the correct order.
More specifically, we let $p(s)=$ the least stage $t\leq
f_{Y^*,Y}(f_{X,X^*}(s+1))$ such that $\ZOP{X}[s]=\ZOP{Y}[t]$, and if $t$
cannot be found then let $p(s)=s$. Similarly, let $q(t)=$ the least stage
$u\leq f_{X^*,X}(f_{Y,Y^*}(t+1))$ such that $\ZOP{Y}[t]=\ZOP{X}[u]$, and if
$u$ cannot be found then let $q(t)=t$.

We first check that $p$ works. Let $w$ and $j$ be such that
\[
\ZeXs[w]=\ZeYs[w]=j.
\]
Let $s$ be the greatest stage such that $\ZOP{X}[s]=j$. Then as $\ZeXs[w]=j$
and $\ZOP{X}[s+1]=j+1$, we certainly have $w<f_{X,X^*}(s+1)$. Then
\[
f_{Y^*,Y}(f_{X,X^*}(s+1))\geq f_{Y^*,Y}(w)
\]
and also
\[
\ZOP{Y}[f_{Y^*,Y}(w)]\geq\ZOP{Y^*}[w]=j.
\]
Therefore, the bound $f_{Y^*,Y}(f_{X,X^*}(s+1))$ is large enough, and we have

\[
\ZOP{X}[s]=\ZOP{Y}[p(s)]=j.
\]
A similar argument holds for $q$.

Now conversely, assume that $p$ and $q$ exist. It is easy to see that we can
make $p$ and $q$ nondecreasing. We wish to show that $(X,Y)$ satisfies the
$\blacklozenge$-property. An obvious candidate for $X^*$ is a set satisfying
$\ZOP{X^*}[p(s)]=\ZOP{X}[s]$ for every $s$, and then we can take $Y^*=Y$, so
that $\ZeXs[p(s)]=\ZeYs[p(s)]$ holds for infinitely many $s$. Unfortunately,
in order to do this, we will need to compute $p^{-1}$ which in general is not
primitive recursive. So we will have to use both $p$ and $q$ to define $X^*$
and $Y^*$.

We call $(s,t)$ a \emph{good pair} if
\[
\ZOP{X}[s]=\ZOP{Y}[p(s)]=\ZOP{Y}[t]=\ZOP{X}[q(t)];
\]
by the hypothesis there are infinitely many good pairs.

First, suppose that there are infinitely many good pairs $(s,t)$ such that
$p(s)\geq s$. Define $c(w)$ to be the largest value of $u\leq w$ such that
$\ZOP{X}[u]\leq\ZOP{Y}[w]$ or $p(u)<w$. Notice that the function $c$ is
primitive recursive and non-decreasing. Therefore, so is the function
$d(w)=\min\left\{d(w-1)+1,\ZOP{X}[c(w)]\right\}$. Furthermore, $d$ has the
property that for any $w$, $d(w+1)\leq d(w)+1$, and that for any $w$ there is
some $t$ satisfying $w\leq t\leq 2w$ such that $d(t)=\ZOP{X}[c(w)]$. (Recall
our convention that for any set $Z$ and any stage $t$,
$\ZOP{Z}[t+1]\leq\ZOP{Z}[t]+1$). Therefore we can define the primitive
recursive set $X^*$ satisfying $\ZOP{X^*}[w]=d(w)$ for all $w$. Take $Y^*=Y$.

Let $\hat{d}(w)$ be the largest value $\leq w$ such that
$\ZOP{X}[\hat{d}(w)]=d(w)$, which is also primitive recursive. Therefore, by
Lemma \ref{lem:reductionintermsofcharfunction} we obviously have
$R_{X^*}\leq_{pr} R_X$. Now we observe that for each $s$, $s\leq c(w)$ where
$w=\max\{s,p(s)+1\}$. Therefore
\[
\ZOP{X^*}[2w]=d(2w)\geq\ZOP{X}[c(w)]\geq\ZOP{X}[s],
\]
showing that $R_{X^*}\geq_{pr} R_X$. Now it follows by a straightforward
induction on $w$ that the following claim is true:
\[
d(w)\geq\max\{\ZOP{X}[u]\mid \ZOP{X}[u]\leq\ZOP{Y}[w]
\]
for some $u\leq w\}$ (using the fact that
\[
\ZOP{X}[c(w)]\geq\max\{\ZOP{X}[u]\mid \ZOP{X}[u]\leq\ZOP{Y}[w]
\]
for some $u\leq w\}$). Now take $(s,t)$ to be a good pair with $p(s)\geq s$.
By the claim above, we have $\ZOP{X^*}[p(s)]\geq \ZOP{X}[s]$. If they were
not equal then $d(p(s))$ would have to be larger than $\ZOP{X}[s]$ which
means that
\[
\ZOP{X}[s]<d(p(s))\leq \ZOP{X}[c(p(s))].
\]
But then as $c(p(s))\geq s$ and $p$ is nondecreasing, we have $p(c(p(s)))\geq
p(s)$, which means, by the definition of $c(p(s))$, that
\[
\ZOP{X}[c(p(s))]\leq\ZOP{Y}[p(s)]=\ZOP{X}[s],
\]
a contradiction. Thus we conclude that
\[
\ZOP{X^*}[p(s)]=\ZOP{X}[s]=\ZOP{Y}[p(s)]=\ZOP{Y^*}[p(s)].
\]

If there are infinitely many good pairs $(s,t)$ such that $q(t)\geq t$ then
we repeat the above, now taking $X^*=X$ and $Y^*$ defined analogously, using
$q$ in place of $p$. So we assume that there are infinitely many good pairs
$(s,t)$ such that $p(s)<s$ and $q(t)<t$. We claim that we can take $X=X^*$
and $Y=Y^*$. Fix a good pair $(s,t)$ such that $p(s)<s$, $q(t)<t$ and

\[
\ZOP{X}[s]=\ZOP{Y}[p(s)]=\ZOP{Y}[t]=\ZOP{X}[q(t)]=j.
\]
Suppose that
\[
\min\{w\mid \ZOP{X}[w]=j\}\leq\min\{w\mid \ZOP{Y}[w]=j\}.
\]
Now this means that
\[
p(s)\geq\min\{w\mid \ZOP{X}[w]=j\}
\]
and since $p(s)<s$ we have that
\[\ZOP{X}[p(s)]=j=\ZOP{Y}[p(s)].
\]
On the other hand if
\[
\min\{w\mid \ZOP{X}[w]=j\}\geq\min\{w\mid \ZOP{Y}[w]=j\}
\]
then
\[
\ZOP{Y}[q(t)]=j=\ZOP{X}[q(t)].
\qedhere
\]
\end{proof}

\begin{corollary}\label{cor:deltapropertycor}
Let $R_X\leq_{pr} R_Y$. Then $(X,Y)$ satisfies the $\blacklozenge$-property
if and only if there exists a primitive recursive function $q$ such that
$\ZOP{Y}[t]=\ZOP{X}[q(t)]$ for infinitely many $t$. Furthermore, we can take
$q$ to be nondecreasing, and we may also replace
``$\ZOP{Y}[t]=\ZOP{X}[q(t)]$" with ``$\ZOP{Y}[t]\leq \ZOP{X}[q(t)]$".
\end{corollary}
\begin{proof}
If $R_X\leq_{pr}R_Y$ then we fix by Lemma
\ref{lem:reductionintermsofcharfunction}, a function $g$ such that
$\ZOP{X}[s]\leq\ZOP{Y}[g(s)]$ for every $s$. But $p(s)\leq g(s)$ for every
$s$, where $p(s)$ is the least such that $\ZOP{X}[s]=\ZOP{Y}[p(s)]$.
\end{proof}

\begin{corollary}\label{cor:deltapropertysubinterval}
If $R_X\leq_{pr}R_Y$ and $(X,Y)$ has the $\blacklozenge$-property then every
subinterval of $(X,Y)$ also has the $\blacklozenge$-property.
\end{corollary}
\begin{proof}
Suppose $R_X\leq_{pr}R_{X'}\leq_{pr}R_{Y'}\leq_{pr} R_Y$. Restricting the
diamond $R_C,R_D$ to the subinterval $(X',Y')$ does not automatically do it,
since for instance, $R_C\vee R_{X'}$ could be above $ R_{Y'}$.

We may assume that $\ZOP{X}[t]=\ZOP{Y}[t]$ for infinitely many $t$. We fix
functions $f$ and $g$ such that for every $t$, $f(t)$ and $g(t)$ are the
least such that $\ZOP{Y'}[t]=\ZOP{Y}[g(t)]$ and $\ZOP{X}[t]=\ZOP{X'}[f(t)]$.
Now given any $t$ let $q(t)=f(u)$ where $u$ is the largest such that
$u<g(t+1)$ and $\ZOP{X}[u]=\ZOP{Y}[u]$. If $u$ cannot be found, let $q(t)=t$.
By Corollary \ref{cor:deltapropertycor} it remains to check that
$\ZOP{Y'}[w]=\ZOP{X'}[q(w)]$ for infinitely many $w$. Suppose that
$\ZOP{X}[t]=\ZOP{Y}[t]=j$ for some $t,j$. Let $w$ be the largest such that
$\ZOP{Y'}[w]=j$. Then
\[
\ZOP{Y}[g(w+1)]=\ZOP{Y'}[w+1]=j+1
\]
and therefore $t<g(w+1)$. By the minimality of $g(w+1)$, we have

\[
\ZOP{X}[u]=\ZOP{Y}[u]=j
\]
for the chosen $u$, and therefore
\[
\ZOP{X'}[q(w)]= \ZOP{X'}[f(u)]=\ZOP{X}[u]=j=\ZOP{Y'}[w].
\qedhere
\]
\end{proof}

Lemma \ref{lem:deltapropertysets} characterizes the $\blacklozenge$-property
in terms of the relative growth rates of $\ZOP{X}$ and $\ZOP{Y}$. For our
next purpose it shall be convenient to express the $\blacklozenge$-property
in terms of the relative growth rates of $p_{\overline{X}}$ and
$p_{\overline{Y}}$. The term ``$n+1$" in the next lemma is important; by
Remark \ref{rem:deltapropertyx1} we cannot replace $p_{\overline{Y}}(n+1)$
with $p_{\overline{Y}}(n)$.

\begin{lemma}\label{lem:deltapropertysetsprincipalfunction}
Let $R_X\leq_{pr} R_Y$. Then $(X,Y)$ satisfies the $\blacklozenge$-property
if and only if there exists a primitive recursive function $r$ such that
$p_{\overline{X}}(n)\leq r(p_{\overline{Y}}(n+1))$ for infinitely many $n$.
\end{lemma}
\begin{proof}
Suppose that $(X,Y)$ has the $\blacklozenge$-property. Fix $q$ as in
Corollary \ref{cor:deltapropertycor}. Let $r(m)=q(u)$, where $u$ is the least
stage such that $Y(i)[u]\downarrow $ for all $i\leq m$. Now let $t$ and $n$
be such that $\ZOP{Y}[t]=\ZOP{X}[q(t)]=n$. Notice that
$p_{\overline{X}}(n)<q(t)$. Let $m=p_{\overline{Y}}(n+1)$ and $u$ be such
that $r(m)=q(u)$. Then since $\ZOP{Y}[t]=n$, we have $t<u$, which means that
$r(m)=q(u)\geq q(t)> p_{\overline{X}}(n)$.

Now suppose that $r$ exists; obviously we may assume that $r$ is
nondecreasing. Define $q(t)$ to be the largest stage $u\leq v$ such that
$\ZOP{X}[u]=\ZOP{Y}[t]$, where $v$ is the least stage such that
$X(j)[v]\downarrow $ for all $j\leq r(t+1)$; if this does not exist, let
$q(t)=t$. Now let $n$ be such that $p_{\overline{X}}(n)\leq
r(p_{\overline{Y}}(n+1))$, and let $t$ be the largest such that
$\ZOP{Y}[t]=n$. We check that $\ZOP{Y}[t]=\ZOP{X}[q(t)]$. We have
$\ZOP{Y}[t+1]=n+1$ and thus $p_{\overline{Y}}(n+1)<t+1$ which means that
\[
r(t+1)\geq  r(p_{\overline{Y}}(n+1))\geq p_{\overline{X}}(n).
\]
This means that $q(t)$ will be equal to some largest $u$ such that
\[
\ZOP{X}[u]=\ZOP{Y}[t]=n.
\]
Hence $\ZOP{X}[q(t)]=\ZOP{Y}[t]$.
\end{proof}

Returning to our question as to whether every interval has the
$\blacklozenge$-property, we can now make use of Lemma
\ref{lem:deltapropertysetsprincipalfunction} to construct an interval that
does not have the $\blacklozenge$-property. In fact, we can show that every
punctual degree is the top of an interval that does not satisfy the
$\blacklozenge$-property:

\begin{proposition}\label{prop:nondeltainterval}
Given any $R_Y$ there is some $R_X<_{pr} R_Y$ such that $(X,Y)$ does not have
the $\blacklozenge$-property.
\end{proposition}
\begin{proof}
We first note that given any total computable function $F$ there is a
coinfinite primitive recursive set $X$ such that $p_{\overline{X}}(n)\geq
F(n)$ for every $n$. Now given any $Y$ we let $F$ to be a computable function
that is fast growing enough so that for every primitive recursive function
$r$, $F(n)>r(p_{\overline{Y}}(n+1))$ for almost all $n$. (Notice that
$p_{\overline{Y}}$ is not necessarily primitive recursive). Now we take $X$
such that $p_{\overline{X}}(n)\geq F(n)$ for all $n$. Then $R_X\leq_{pr} R_Y$
by Proposition \ref{prop:growth-rate} and $(X,Y)$ does not have the
$\blacklozenge$-property by Lemma
\ref{lem:deltapropertysetsprincipalfunction}. This of course implies that
$Y\nleq_{pr}X$.
\end{proof}

On the other hand, it is not the case that every punctual degree is the
bottom of an interval that does not have the $\blacklozenge$-property. For
instance, if $(X,T)$ has the $\blacklozenge$-property then by Corollary
\ref{cor:deltapropertysubinterval} there is no $Y$ such that
$R_X\leq_{pr}R_Y$ and $(X,Y)$ does not have the $\blacklozenge$-property.

By Corollary \ref{cor:deltapropertydefinable}, the $\blacklozenge$-property
is definable in the language of partial orders. This property will be crucial
in our subsequent analysis of $\Peq$. The next most natural step when
studying a new degree structure is to verify whether the degree structure is
rigid or homogeneous. We introduce an important definition that shall soon
prove to be very useful:

\begin{defi}Given any primitive recursive set $X$, we define $X^{[-1]}$ to be the set defined by:
\[X^{[-1]}(n)=\begin{cases}
1, &\text{if $n$ is the least such that $X(n)=0$},\\
X(n),&\text{otherwise}.
\end{cases}\]
In particular, $\ZOP{X^{[-1]}}[t]=\max\{0,\ZOP{X}[t]-1\}$ for every stage $t$.
\end{defi}

An immediate consequence of the definition is:

\begin{lemma}\label{lem:delta13}
$R_{X^{[-1]}}<_{pr} R_X$ if and only if $R_X<_{pr} \Id$.
\end{lemma}
\begin{proof}
Since $\ZOP{X^{[-1]}}[t]\leq \ZOP{X}[t]$ for every $t$, we have
$R_{X^{[-1]}}\leq_{pr} R_X$ (by Lemma
\ref{lem:reductionintermsofcharfunction}), so we have to show that
$R_{X^{[-1]}}\geq_{pr} R_X$ if and only if $R_X\geq_{pr} \Id$. Suppose that
$g$ reduces $\Id$ to $R_X$. Let $n$ be the least element not in $X$. If
$n\not\in \textrm{rng}(g)$ then $g$ is already a reduction from $\Id$ to
$R_{X^{[-1]}}$, and if $g(m)=n$ then the function $h(k)=g(k+m+1)$ will reduce
$\Id$ to $R_{X^{[-1]}}$.

Conversely suppose that $f$ reduces $R_X$ to $R_{X^{[-1]}}$. Define the
function $F$ by the following. Let $F(0)=\max f([0,n])$ where $n$ is the
second element not in $X$. Let $F(k+1)=\max f([0,F(k)])$. Then for each $k$,
there are at least $k+1$ many distinct elements not in $X$ which are smaller
than $F(k)$. The function $F$ can easily be used to define a reduction of
$\Id$ to $R_X$.
\end{proof}

Our first question about rigidity is easily answered by Lemma
\ref{lem:delta13}:

\begin{theorem}\label{thm:nonrigidity}
$(\Peq,\leq)$ is not rigid.
\end{theorem}
\begin{proof}
The map $\deg(R_X)\mapsto \deg\left(R_{X^{[-1]}}\right)$ is a non-trivial
automorphism. In fact, it fixes $\deg(\Id)$ and moves every other degree to a
strictly smaller degree.
\end{proof}

\begin{corollary}\label{cor:finitenotdefinable}
The only definable degree is the greatest degree, $\deg(\Id)$. No finite set
of degrees is definable except for $\{\deg(\Id)\}$.
\end{corollary}

We turn now our attention to the question of how homogeneous the structure
$(\Peq,\leq)$ is. (Un)fortunately, the structure $(\Peq,\leq)$ is neither
rigid nor homogeneous, which indicates that the structure is not as trivial
as might seem at first glance. This justifies further investigations into the
degree structure of $\Peq$.

\begin{defi}
We call a coinfinite primitive recursive set $X$ \emph{slow} if for every
primitive recursive function $r$,
$p_{\overline{X}}(n+1)>r(p_{\overline{X}}(n))$ holds for almost every $n$.
\end{defi}
A slow set generates its zeros slower than any primitive recursive recursive
function can predict (infinitely often). Slow sets obviously exist. An
immediate consequence of Lemma \ref{lem:deltapropertysetsprincipalfunction}
is:

\begin{corollary}\label{cor:slowimpliesnotdelta}
If $X$ is slow then $(X,\Id)$ does not satisfy\footnote{Strictly speaking, we
should write $(X,2\omega)$ instead of $(X,\Id)$ here.}  the
$\blacklozenge$-property.
\end{corollary}
Thus, if $X$ is slow and $(Y,\Id)$ satisfies the $\blacklozenge$-property
($Y$ exists, by Theorem \ref{thm:density}), then no automorphism of
$(\Peq,\leq)$ can map $\deg(Y)$ to $\deg(X)$. This fact also means that
uppercones of $\Peq$ are not always isomorphic to each other.

We also note that the converse to Corollary \ref{cor:slowimpliesnotdelta} is
false: Given any $X$ we can easily find some $Y$ such that
$R_{X^{[-1]}}\leq_{pr}R_Y\leq_{pr}R_X$ and $Y$ is not slow; to do this we can
arrange for $p_{\overline{Y}}(n+1)=p_{\overline{Y}}(n)+1$ for infinitely many
$n$. Then for each such $Y$, $(Y,\Id)$ cannot satisfy the
$\blacklozenge$-property, by Corollary \ref{cor:deltapropertysubinterval}.

When we turn to lowercones however, the situation is less obvious. Since
every punctual degree is the top of an interval with the
$\blacklozenge$-property as well as the top of (another) interval without the
$\blacklozenge$-property, it is not clear how we can immediately distinguish
two lowercones from each other using the $\blacklozenge$-property, similarly
to how we separated uppercones. In fact, the lowercone $\{\deg(R_Y)\mid
R_Y\leq_{pr} R_X\}$ is isomorphic to the lower cone $\{\deg(R_Y)\mid
R_Y\leq_{pr} R_{X^{[-1]}}\}$.

Hence it is entirely conceivable that every lowercone $\{\deg(R_Y)\mid
R_Y\leq_{pr} R_X\}$ is isomorphic to $(\Peq,\leq)$. From the point of view of
each degree of $R_Y$ where $R_Y\leq_{pr}R_X$, the set $X$ has no delay, since
the zeros of $X$ are always generated no slower than the zeros of $Y$. Hence
we might expect to always be able to extend any partial embedding of $\Peq$
into $\{\deg(R_Y)\mid R_Y\leq_{pr} R_X\}$. We will show that this is not the
case. The key to our analysis lies (again!) in the operator $X\mapsto
X^{[-1]}$.

\begin{rem}\label{rem:deltapropertyx1}
By Lemma \ref{lem:deltapropertysetsprincipalfunction}, $(X^{[-1]},X)$ will
have the $\blacklozenge$-property for any $X$. Consequently, we cannot
replace ``$p_{\overline{Y}}(n+1)$" in Lemma
\ref{lem:deltapropertysetsprincipalfunction} with ``$p_{\overline{Y}}(n)$";
for instance, if $Y$ is slow and $X=Y^{[-1]}$.
\end{rem}

Even though an interval of the form $(X^{[-1]},X)$ will always satisfy the
$\blacklozenge$-property, the same is not true of an interval of the form
$(X^{[-2]},X)$, where $X^{[-2]}=\left(X^{[-1]}\right)^{[-1]}$.
\begin{lemma}\label{lem:x2deltaproperty}
$(X^{[-2]},X)$ satisfies the $\blacklozenge$-property if and only if $X$ is
not slow.
\end{lemma}

\begin{proof}
We apply Lemma~\ref{lem:deltapropertysetsprincipalfunction} and noting that
$p_{\overline{X^{[-2]}}}(n)=p_{\overline{X}}(n+2)$.
\end{proof}

\begin{lemma}\label{lem:x1boundsallnondeltaproperty}
Given any $R_Y\leq_{pr} R_X$, either $(Y,X)$ satisfies the
$\blacklozenge$-property or $R_Y\leq_{pr}R_{X^{[-1]}}$.
\end{lemma}
\begin{proof}
If $R_Y\nleq_{pr}R_{X^{[-1]}}$ then $\ZOP{Y}[s]>\ZOP{X^{[-1]}}[s]$ for
infinitely many $s$, which means that $\ZOP{Y}[s]\geq \ZOP{X}[s]$ for
infinitely many $s$. Apply Corollary \ref{cor:deltapropertycor}.
\end{proof}

Lemma~\ref{lem:x1boundsallnondeltaproperty} tells us that $R_{X^{[-1]}}$
bounds all $R_Y$ below $R_X$ such that $(Y,X)$ does not have the
$\blacklozenge$-property. This will allow us to define the map
$\deg(R_X)\mapsto \deg(R_{X^{[-1]}})$. Towards this, we prove another lemma:

\begin{lemma}\label{lem:definablelem1}
Let $R_X$ and $R_Y$ be punctual equivalence relations with
$R_{X^{[-1]}}\nleq_{pr} R_Y$. Then there is some $Z$ such that $R_Z\leq_{pr}
R_{X^{[-1]}}$, $(Z,X)$ does not have the $\blacklozenge$-property and
$R_Z\nleq_{pr} R_Y$.
\end{lemma}

\begin{proof}
Fix a computable listing $\{p_e\}_{e\in\omega}$ of all primitive recursive
functions as in Section~\ref{ssct:listing}. By making the function values
larger, we may assume that for every $p_e$ is strictly increasing, and that
$p_{e+1}(n)>p_e(n)$. This listing $(e,n)\mapsto p_e(n)$ is total computable
but of course not primitive recursive. We will also assume that
$p_e(x)=\psi(x)$ for some total computable function which halts in fewer than
$\hat{p}_e(x)$ many steps, where $\hat{p}_e$ is some primitive recursive
function. All indices can be found effectively.

We now define the set $Z$ in stages. Since $Z$ must be primitive recursive,
at every stage $s$ we must decide $Z\upharpoonright s+1$. In the below
construction, at each stage $s+1$, we will declare $\ZOP{Z}[s+1]=\ZOP{Z}[s]$
or $\ZOP{Z}[s]+1$; in the former case we mean that we set $Z(s+1)=1$ and in
the latter case we set $Z(s+1)=0$.

At stage $s$ we will have a parameter $V(s)$ which stands for the index such
that at stage $s$ we are attempting to make
$\ZOP{Z}[s]\nleq\ZOP{Y}[p_{V(s)}(s)]$. At stage $s=0$ we declare $0\in Z$
(i.e. $\ZOP{Z}[0]=0$) and set $V(0)=0$. Suppose we have the value of
$\ZOP{Z}[s]$ and $V(s)=e$. Compute $p_e(k)$ for one more step (where $k$ was
the input that was last processed).  If there is a new convergence $p_e(k)$
seen at this step such that $\ZOP{X}[k]\neq\ZOP{X}[k+1]$, we take
\[
\ZOP{Z}[s+1]=\min\{\ZOP{Z}[s]+1,\ZOP{X}[s]-1\}.
\]
Check if $\ZOP{Z}[t]>\ZOP{Y}[p_e(t)]$ for any $t\leq s$ for which we have
already found the value of $p_e(t)$. If so we increase $V$ by one. In all
other cases take $\ZOP{Z}[s+1]=\ZOP{Z}[s]$, and go to the next stage with the
same value of $V$.

The above gives a primitive recursive description of the set $Z$. Since
$\ZOP{Z}[s]<\ZOP{X}[s]$ for every $s$, we have $R_Z\leq_{pr} R_{X^{[-1]}}$.
Notice that the construction processes the inputs $k$ sequentially; namely,
the construction begins with $k=0$ and $V=0$ and waits for $p_0(0)$ to
converge (this takes $\hat{p}_0(0)$ many stages). It then moves on to $k=1$
and waits for $p_0(1)$ to converge, and so on. If ever, the construction
decides to increase the value of $V$ while waiting on, say, $p_0(5)$, then we
will move on to wait for $p_1(5)$ to converge, then $p_1(6)$, and so on. Let
$k^*(s)$ be the value of $k$ being processed by the construction at stage
$s$. Since $\{p_e\}_{e\in\omega}$ are all total, $\lim_s p^*(s)=\infty$.
Define the sequence $\{k_i\}_{i\in\omega}$ such that $\ZOP{X^{[-1]}}[k_i]=i$
and $\ZOP{X^{[-1]}}[k_i+1]=i+1$. Take $k_{-1}=-1$.

\begin{claim}\label{claim1}
For every stage $s$ we have $\ZOP{Z}[s]=i$, where $k_{i-1}<k^*(s)\leq k_i$.
\end{claim}
\begin{proof}
If $s=0$ then $i=k^*(0)=0$ and so $\ZOP{Z}[0]=0$. Assume $\ZOP{Z}[s]=i$ where
$k_{i-1}<k^*(s)\leq k_i$. Since the value of $\ZOP{Z}[s+1]$ is decided at the
end of stage $s$, we have to examine what the construction did at stage $s$.
At stage $s$ we would increase the value of $\ZOP{Z}$ only if
$p_{V(s)}(k^*(s))$ is found to converge at that stage and $k^*(s)=k_i$. In
that case $k^*(s+1)=k_i+1$ and so $k_i<k^*(s+1)\leq k_{i+1}$, and so we have
to check that $\ZOP{Z}[s+1]=i+1$. But note that as $k^*(s)< s$ we have
\[
\ZOP{X^{[-1]}}[s]\geq\ZOP{X^{[-1]}}[k^*(s+1)]=i+1=\ZOP{Z}[s]+1
\]
and so
\[
\ZOP{Z}[s+1]=\min\{\ZOP{Z}[s]+1,\ZOP{X}[s]-1\}=i+1. \qedhere
\]
\end{proof}

Next, we verify that $\lim_s V(s)=\infty$; suppose not. Let $t_0$ be the
least such that $V(t_0)=e$ and $V(s)=e$ for almost all $s$. Let
$t_0<t_1<t_2<\cdots$ be the stages such that $p_e(l_i)$ first converged at
stage $t_i$, where $i>0$, $k^*(t_i)=l_i<k^*(t_i+1)$ and $l_i=k_{i+t_0}$. Note
that $l_1> t_0$. By our convention above, we have that $t_i=\hat{p}_e(l_i)$.
By Claim \ref{claim1} we see that $\ZOP{Z}[t]=\ZOP{Z}[t_{i+1}]$ for every $t$
and $i$ such that $t_i<t\leq t_{i+1}$.

For every $k$ and $i>0$ such that $l_i<k\leq l_{i+1}$, we have
$t_i+1=\hat{p}_e(l_i)+1\leq\hat{p}_e(l_i+1)\leq \hat{p}_e(k)\leq
\hat{p}_e(l_{i+1})=t_{i+1}$, and since $\ZOP{Z}[t_i+1]=\ZOP{Z}[t_{i+1}]$, we
also have $\ZOP{Z}[\hat{p}_e(k)]=\ZOP{Z}[t_{i+1}]$. Therefore, we have
\begin{align*}
\ZOP{X^{[-1]}}[k]&\leq \ZOP{X^{[-1]}}[l_{i+1}] \quad \text{(by Claim \ref{claim1})}\\
&=\ZOP{Z}[t_{i+1}]\\
&=\ZOP{Z}[\hat{p}_e(k)] \quad \text{(as the construction never increased $V$ after $t_0$)}\\
&\leq \ZOP{Y}[p_e(\hat{p}_e(k))].
\end{align*}
This shows that $R_{X^{[-1]}}\leq_{pr}R_{Y}$, contrary to our assumption.

Now that we know $\lim_s V(s)=\infty$, for every $e$ there must be a stage
$s$ of the construction where we saw $\ZOP{Z}[t]>\ZOP{Y}[p_e(t)]$ for some
$t\leq s$, which means that $R_Z\nleq_{pr} R_Y$. Now consider some primitive
recursive function $p_e$. Let $V(s)=e$ for some $s$. For each $k>k^*(s)$ we
let $t>s$ be the stage where $k^*(t)=k<k^*(t+1)$, and $i$ be such that
$k_{i-1}<k\leq k_{i}$. By Claim \ref{claim1}, $\ZOP{Z}[t]=i$. We now have
\begin{align*}
\ZOP{Z}[p_e(k)]&\leq \ZOP{Z}[\hat{p}_e(k)]\\
&\leq\ZOP{Z}[\hat{p}_{V(t)}(k)] \quad &\text{(at stage $t$ we saw $p_{V(t)}(k)$ converge)}\\
&=\ZOP{Z}[t]\\
&= i\\
&=\ZOP{X^{[-1]}}[k_i]\\
&=\ZOP{X^{[-1]}}[k]\\
&<\ZOP{X}[k].
 \end{align*}
Thus, $(Z,X)$ does not have the $\blacklozenge$-property.
\end{proof}

Now we are ready to apply the analysis started above.

\begin{corollary}
The map $\deg(R_X)\mapsto \deg(R_{X^{[-1]}})$ is definable in $(\Peq,\leq)$.
\end{corollary}
\begin{proof}
Apply Lemmas \ref{lem:x1boundsallnondeltaproperty} and
\ref{lem:definablelem1} to see the following. Given punctual $R_X$ and $R_Y$,
we have $R_Y\equiv_{pr} R_{X^{[-1]}}$ if and only if the following hold:
\begin{enumerate}
\item  $R_Y\leq_{pr} R_X$,
\item For every $R_Z\leq_{pr} R_X$ such that $(Z,X)$ does not have the
    $\blacklozenge$-property, $R_Z\leq_{pr} R_Y$, and
\item If $R_V$ has properties (1) and (2), then $R_V\geq_{pr} R_Y$.
\end{enumerate}
That is, we may define $R_{X^{[-1]}}$ as the least degree below $R_X$ that is
an upperbound for the set of all degrees $R_Z\leq_{pr} R_X$ such that $(Z,X)$
does not have the $\blacklozenge$-property. Since the
$\blacklozenge$-property is definable, so is the set of all pairs $\left(
\deg(R_X),\deg(R_{X^{[-1]}}) \right)$.
\end{proof}

\begin{corollary}
If $\psi$ is an automorphism of $(\Peq,\leq)$, then for any $R_X$, we have
that $\psi(\deg(R_{X^{[-1]}}))=\deg(R_{Y^{[-1]}})$, where
$R_Y\equiv_{pr}\psi(R_X)$.
\end{corollary}

\begin{corollary}
If $\psi$ is an automorphism of $(\Peq,\leq)$, then $R_X$ is slow if and only
if $\psi(R_X)$ is slow.
\end{corollary}

\begin{proof}
By Lemma \ref{lem:x2deltaproperty}, $R_X$ is slow if and only if
$(X^{[-2]},X)$ does not have the $\blacklozenge$-property, if and only if
$\left(\psi(X^{[-2]}),\psi(X)\right)$\footnote{The notation is not formally
correct, but has the obvious meaning here.} does not have the
$\blacklozenge$-property. But then $\psi(X^{[-2]})=\psi(X)^{[-2]}$ which is
equivalent to the fact that $\psi(X)$ is slow.
\end{proof}

We are now able to give a negative answer to our question above regarding
whether every lowercone is isomorphic to $\Peq$. In fact, \emph{no} proper
lowercone can be isomorphic to $\Peq$. Even though each lowercone is
principal, our intuition that we should be able to replicate the structure
below $\Id$ in any lowercone by ``relativising" is entirely incorrect.

\begin{corollary}
The lowercone $\{\deg(R_Y)\mid R_Y\leq_{pr}R_X\}$ below $R_X$ can never be
isomorphic to $(\Peq,\leq)$ unless $R_X\equiv_{pr}\Id$. Thus, no proper
lowercone can be isomorphic to $(\Peq,\leq)$.
\end{corollary}
\begin{proof}
If $R_X<_{pr}\Id$ then by Lemma \ref{lem:delta13}, $R_{X^{[-1]}}<_{pr}R_X$.
Thus, by Lemma \ref{lem:x1boundsallnondeltaproperty}, $\deg(R_{X^{[-1]}})$ is
a degree strictly below $\deg(R_X)$ that is an upperbound on the set of all
degrees  $R_Z\leq_{pr} R_X$ such that $(Z,X)$ does not have the
$\blacklozenge$-property (and thus does not embed the diamond). By Lemma
\ref{lem:definablelem1}, (applied with $R_X=R_{X^{[-1]}}=\Id$), no degree
strictly below $\deg(\Id)$ can serve as the image of $\deg(R_{X^{[-1]}})$.
\end{proof}

Our analysis on lowercones exploits the unique property satisfied by
$X^{[-1]}$, and thus can only be applied to separate an incomplete (or
proper) lowercone from $\Peq$. Using our analysis, we can still say a little
more; we show that not every pair of proper lowercones are isomorphic:
\begin{corollary}
If $X$ is slow and $Y$ is not, then their lowercones cannot be isomorphic (as
posets).
\end{corollary}
\begin{proof}
Because $X^{[-1]}$ (and similarly $Y^{[-1]}$) is the least upperbound of the
set of all $R_Z\leq_{pr}R_X$ such that $(Z,X)$ does not embed the diamond,
any isomorphism between the two lowercones must send $X^{[-1]}$ to $Y^{[-1]}$
and therefore must map $X^{[-2]}$ to $Y^{[-2]}$. However, as $X$ is slow and
$Y$ is not, by Lemma \ref{lem:x2deltaproperty}, $(Y^{[-2]},Y)$ satisfies the
$\blacklozenge$-property whereas $(X^{[-2]},X)$ does not.
\end{proof}
Since it is not hard to construct a pair of incomparable degrees, one of
which is slow and the other is not, we immediately have a pair of
incomparable lowercones that are not isomorphic. This leaves the intriguing
question as to whether any pair of incomparable lowercones are isomorphic. We
leave this question open:
\begin{question}
Are there incomparable degrees $R_X\mid_{pr} R_Y$ with isomorphic lowercones?
\end{question}

\begin{question}
Are there continuum many automorphisms of $(\Peq,\leq)$?
\end{question}

Given that (Corollary \ref{cor:finitenotdefinable}) no finite set of degrees
except for $\{\deg(\Id)\}$ is definable (without parameters), it may be
difficult to apply the analysis of \cite{Andrews-Sorbi-19,Andrews-et-al-20}
to encode finite graphs into $\Peq$, which would involve having to define
finite sets of degrees, albeit with a parameter. So we also ask:
\begin{question}
Is the first order theory of  $(\Peq,\leq)$ decidable?
\end{question}

\bibliographystyle{plain}
\bibliography{prim-rec}

\end{document}